\documentclass[11pt]{amsart}

\usepackage{amsmath,amsfonts,amssymb,amsthm,mathrsfs}
\usepackage{hyperref}
\usepackage{enumitem}
\hypersetup{
    colorlinks=true,
    linkcolor=blue,
    citecolor=blue,
    urlcolor=blue
}
\newtheorem{definition}{Definition}[section]
\newtheorem{theorem}[definition]{Theorem}
\newtheorem{proposition}[definition]{Proposition}
\newtheorem{lemma}[definition]{Lemma}
\newtheorem{remark}[definition]{Remark}
\title[Spherical Milnor Spaces and Topological Structures]{
Spherical Milnor Spaces II: Projective Quotients and Higher Topological Structures
}

\author{Jean-Pierre Magnot}

\address{{SFR MATHSTIC, LAREMA, Universit\'e d’Angers, 2 Bd Lavoisier, 
49045 Angers cedex 1, France;  Lyc\'ee Jeanne d'Arc, 40 avenue de Grande Bretagne, 63000 Clermont-Ferrand, 
France}; Lepage Research Institute, 17 novembra 1, 081 16 Presov, Slovakia}
\email{\small magnot@math.cnrs.fr; jean-pierr.magnot@ac-clermont.fr; jp.magnot@gmail.com}



\begin{document}

\begin{abstract}
We study projective and semidirect quotients of the spherical
Milnor construction and develop a step-by-step lifting
obstruction from their explicit spherical local sections.
The resulting transition functions produce a twisted Cech
2-cocycle for any specified compatible extension with abelian
kernel.  When the kernel is $U(1)$ and the sign acts by
complex conjugation, taking logarithms yields a twisted
Dixmier--Douady obstruction under stated sheaf hypotheses.
We identify its projective-operator interpretation when an
additional representation is supplied, without attempting
a classification of K-theory twists.  We also investigate
descent of differential-geometric structures and, under
restricted-operator hypotheses, Schwinger cocycles and
lifting gerbes for loop groups.
\end{abstract}

\maketitle
\tableofcontents
\section{Introduction}

The construction of classifying spaces is one of the central organizing
principles of modern topology and geometry. Since Milnor's construction of
universal bundles \cite{Milnor1956I,Milnor1956II}, classifying spaces have
provided a natural framework for the study of principal bundles,
characteristic classes, and homotopical invariants. The later developments of
Segal, Dold, Rector, Dupont, Kamber--Tondeur and Karoubi placed this circle of
ideas in a broader homotopical, cohomological and \(K\)-theoretic framework
\cite{Segal1968,Dold1963,Rector1971,Dupont1978,KamberTondeur1975,Karoubi1978}.
In particular, the interaction between classifying spaces, local coefficient
systems, twisted characteristic classes and generalized cohomology theories is
now a standard but still very active part of global geometry.

The present paper is the second part of a program devoted to spherical
versions of Milnor classifying spaces. In Part~I \cite{Magnot-MilnorI}, we
introduced the spherical Milnor model
\[
E^{\mathrm{sph}}(G),
\]
defined by replacing the usual barycentric relation
\[
\sum_i t_i=1,\qquad t_i\geq0,
\]
by the quadratic normalization
\[
\sum_i x_i^2=1.
\]
This change has substantial geometric consequences. The resulting space is not
merely a reparametrization of Milnor's original model: its diffeology differs
from the Milnor--Watts diffeology because its locally finite generating plots are taken from spheres
with their subset diffeology, whereas Magnot--Watts use a different
product-and-quotient diffeology on barycentric coordinates. The diffeology of Milnor's classifying space
was studied in \cite{MagnotWatts2017},
within the broader framework of diffeology developed by Souriau and
Iglesias-Zemmour \cite{Souriau1980,IglesiasZemmour2013}. The present paper
builds on that analysis and focuses on projective quotients and explicit lifting
obstructions, with an optional K-theoretic interpretation of
particular projective realizations.

Diffeology provides a convenient setting for generalized smooth spaces beyond
manifolds. It is flexible enough to include quotients, mapping spaces,
singular spaces, and infinite-dimensional examples, while retaining a usable
notion of smooth parametrization. Foundational and comparative aspects of
diffeology appear in \cite{Hector1995,Stacey2011,Watts2012Thesis,GMW2023}.
The \(D\)-topology and homotopy theory of diffeological spaces were developed
in \cite{ChristensenSinnamonWu2014,ChristensenWu2014Homotopy}, and smooth
classifying spaces were studied in \cite{ChristensenWu2021Classifying}. The
tangent, exterior, and vector-space aspects of diffeology play a role in the
background of this work and are treated in
\cite{ChristensenWu2016Tangent,ChristensenWu2017Filtered,ChristensenWu2019Vector,ChristensenWu2022Exterior}.
Vector pseudo-bundles, pushforwards, and diffeological connections provide the
appropriate framework for the geometric constructions used in Part~I
\cite{Pervova2016,Pervova2017,Pervova2018,Wu2023}. Related
views on smooth generalized spaces and infinite-dimensional calculus are also
found in convenient analysis and synthetic differential geometry
\cite{KrieglMichor1997,Kock2006}. The example of the irrational torus shows
how diffeology may interact with noncommutative geometric phenomena
\cite{IglesiasLaffineurZemmour2020}.

The main aim of the present article is to study the topological structures
naturally attached to the spherical Milnor construction. A key feature is the
presence of a canonical \(\mathbb Z_2\)-symmetry
\[
(x_i,g_i)\mapsto (-x_i,g_i),
\]
which leads to projective and twisted quotients. When this symmetry interacts
with an action
\[
\sigma:\mathbb Z_2\to \operatorname{Aut}(G),
\]
one obtains semidirect products
\[
H_\sigma=G\rtimes_\sigma\mathbb Z_2
\]
and corresponding twisted spherical classifying spaces
\[
B^{\mathrm{sph}}_\sigma(G).
\]
The associated transition functions naturally split into a
\(\mathbb Z_2\)-valued cocycle and a \(G\)-valued cocycle twisted by the local
system determined by the first one. This leads to a non-abelian cocycle
picture, closely related to the classical theory of non-abelian cohomology
\cite{Giraud1971} and to higher-categorical approaches to non-abelian
topology \cite{BrownHiggins1981,BrownHigginsSivera2011,BaezLauda2004}.

A central part of the paper is devoted to lifting problems and gerbes. Given
an extension
\[
1\to A\to \widehat H_{\widehat\sigma}\to H_\sigma\to 1,
\]
a principal \(H_\sigma\)-bundle has, in general, an obstruction to being lifted
to a principal \(\widehat H_{\widehat\sigma}\)-bundle. In the presence of the
\(\mathbb Z_2\)-twist, this obstruction lives not in an ordinary coefficient
system but in the sheaf of smooth sections of an associated
$A$-bundle, twisted by the sign action and the projective class
\[
\alpha\in H^1(X;\mathbb Z_2).
\]
The corresponding lifting gerbe is banded by \(A_\alpha\), and its class lies
in
\[
H^2(X;\mathscr A_\alpha).
\]
This point of view connects the present construction with the classical and
modern theories of gerbes, bundle gerbes, and higher stacks
\cite{Breen1994,Breen2008,BreenMessing2005,Murray1996,Stevenson2000,Waldorf2007,NikolausWaldorf2013}.
The language of stacks and orbifold-like quotient structures is also relevant
to the background of these constructions \cite{MoerdijkPronk1997,Noohi2005,Noohi2012}.
The higher bundle and \(\infty\)-geometric viewpoint developed in
\cite{NikolausSchreiberStevenson2012,Schreiber2013,FiorenzaSchreiberStasheff2007}
provides a natural conceptual extension of the present picture. In particular,
the relation between differential cocycles, higher principal bundles and
higher characteristic classes motivates the interpretation of the obstruction
classes introduced below. Waldorf's transgression-regression methods and
models for string \(2\)-groups also provide a useful comparison point
\cite{Waldorf2010}, while Wockel's work relates such questions to classifying
spaces of principal \(2\)-bundles \cite{Wockel2008}.

The principal geometric contribution of this part is an explicit
construction of lifting obstructions directly from spherical
coordinate charts.  First we normalize a nonzero spherical
coordinate to obtain local sections of the quotient.  Their
transitions split into sign and group components, and lifts
through a specified extension may fail to satisfy the
cocycle identity.  This defect gives a degree-two obstruction.
In the $U(1)$ case with conjugation action, its local
logarithms yield a degree-three integral cocycle with sign
coefficients, under suitable sheaf-theoretic hypotheses.
The obstruction is not a classification of all possible
Dixmier--Douady or K-theory twists: an extension is an
additional datum, and a K-theoretic interpretation requires
an additional projective realization.  Ordinary complex
twisted K-theory enters only as a comparison in the cases
where such a realization exists \cite{AtiyahSegal2004,Rosenberg1989}.

Another motivation comes from operator-theoretic defects. In Part~I, Hodge-
and Dirac-type structures were constructed on the spherical Milnor model under
appropriate finite-rank or analytic assumptions. Such structures need not
descend to the quotient unless they are compatible with the group action. The
failure of descent defines a defect. In favorable restricted Hilbert settings,
this defect can be related to Schwinger cocycles, central extensions and
lifting gerbes. The classical model for this phenomenon is the theory of loop
groups and restricted Grassmannians \cite{PressleySegal1986}, together with
the appearance of gerbes and anomalies in index theory and Hamiltonian
quantization \cite{CareyMickelssonMurray1997,Freed1995}. Abelian and
non-abelian extensions of infinite-dimensional Lie groups provide the natural
background for these constructions \cite{Neeb2004,Neeb2007}. In the present
article, we formulate the Schwinger-cocycle interpretation carefully as a
conditional construction: it requires a polarization, restricted operators, a
well-defined Schwinger cocycle, and an integrability condition to pass from
Lie algebra cocycles to group extensions.

The paper is organized as follows. Section~\ref{sec:prelim} recalls the
diffeological Milnor and spherical Milnor constructions needed in the sequel.
We emphasize the difference between the simplicial and spherical diffeologies,
the role of the projective \(\mathbb Z_2\)-quotient, and the relation with
Part~I. The next sections study strata, smoothness criteria, and the behavior
of forms and maps with respect to the final diffeology. We then introduce
twisted spherical classifying spaces associated with semidirect products
\[
G\rtimes_\sigma \mathbb Z_2
\]
and describe their transition cocycles. The subsequent sections develop the
lifting gerbe associated with an extension of \(H_\sigma\), prove that its
class lies in the twisted group
\[
H^2(X;\mathscr A_\alpha),
\]
and identify the corresponding twisted Dixmier--Douady invariant when
\(A=U(1)\). We then discuss the associated twists of \(K\)-theory. Finally, we
analyze the descent of geometric structures, give the operator
cocycle identity for equivariance defects, and explain the additional
restricted-representation hypotheses needed for Schwinger cocycles,
central extensions and gerbes.

The main results organize the degree-one sign class and the
lifting obstructions into a common spherical classifying framework:
\[
\alpha\in H^1(X;\mathbb Z_2),\qquad
\beta\in H^2(X;\mathscr A_\alpha),\qquad
\operatorname{DD}_\alpha\in H^3(X;\mathbb Z_\alpha),
\]
where the latter class requires the conjugation action on $U(1)$
and the stated sheaf-theoretic hypotheses. The analysis also gives
precise sufficient conditions for geometric descent, projective
realizations of lifting obstructions and Schwinger lifting gerbes, while identifying
obstructions to naive global Hodge--Dirac and higher-dimensional
current-group extensions.
\section{Preliminaries on diffeological Milnor constructions}
\label{sec:prelim}

We briefly recall the constructions and notions developed in Part~I \cite{Magnot-MilnorI} that will
be used throughout the present article. General references on diffeology are
\cite{IglesiasZemmour2013,GMW2023}.

\subsection{Diffeological spaces and quotient constructions}

A diffeological space is a set endowed with a family of parametrizations
(called plots) stable under smooth reparametrization and satisfying locality
and covering conditions. A map between diffeological spaces is smooth if it
pulls back plots to plots.

Given a family of maps
\[
f_i:X_i\to X,
\]
the final diffeology on \(X\) is the coarsest diffeology making all \(f_i\)
smooth. Quotient diffeologies are special cases of final diffeologies.

A diffeological group is a group equipped with a compatible diffeology for
which multiplication and inversion are smooth.

\subsection{Milnor classifying spaces and the precise diffeology}
\label{subsec:partII-mw}
For a diffeological group $G$, the original Magnot--Watts total
space $E_{\rm MW}(G)$ consists of finitely supported formal barycentric
sums with labels in $G$, where labels at zero weights are forgotten.
The generating barycentric coordinates lie in
$[0,1]^{\mathbb N}$ with its product diffeology (each factor has
its ordinary subset diffeology), and the construction uses the
corresponding subset and quotient diffeologies; see
\cite[Definition~3.1]{MagnotWatts2017}.  This is \emph{not} the
finite-chart final diffeology used below.

For comparison, one may also put the final diffeology on the same
underlying set generated by all maps
$\Delta_I\times G^I\to E_{\rm LF}(G)$, where $I$ ranges over finite
nonempty subsets of $\mathbb N$ and
\[
 \Delta_I=\{t\in[0,1]^I:\ \textstyle\sum_{i\in I}t_i=1\}
\]
has its usual subset diffeology.  The identity
$E_{\rm LF}(G)\to E_{\rm MW}(G)$ is smooth, but an arbitrary
Magnot--Watts plot need not have locally bounded support, so the
converse need not be smooth.  We shall not conflate these spaces.

\begin{remark}[No boundary stationarity axiom]
A plot $t(s)=s^2$ into $[0,1]$ has $t'(0)=0$ but is not locally
constant at $0$.  Thus the subset diffeology does not impose
stationarity near zeros.  A particular endpoint-stationary
reparametrization used to concatenate homotopies must not be
mistaken for a condition on all plots of $E_{\rm MW}(G)$.
\end{remark}

The classical Magnot--Watts bundle $E_{\rm MW}(G)\to
B_{\rm MW}(G)$ is a weakly $D$-numerable principal bundle and
has a classifying property with the hypotheses described in
\cite{MagnotWatts2017}.  The corresponding property for the
spherical construction is treated separately in Section~5.

\subsection{The spherical Milnor model}
Let
\[
 S_I=\{x\in\mathbb R^I:\textstyle\sum_{i\in I}x_i^2=1\}
\]
and define $E^{\rm sph}(G)$ as the set of finite-support pairs
$(x_i,g_i)$ modulo forgetting $g_i$ whenever $x_i=0$.  It carries
the final diffeology generated by the maps
\[
 \psi_I:S_I\times G^I\longrightarrow E^{\rm sph}(G),
\]
where the coordinates are extended by zero outside $I$.  Every
plot therefore has, locally in its parameter domain, a lift
through a finite-index generating chart.  If $G$ is Fr\'echet,
the charts are generally infinite-dimensional despite the finite
index set $I$.

\subsection{Comparison of the two models}
\label{subsec:partII-comparison}
\begin{proposition}
The formula
\[
 \Phi:E^{\rm sph}(G)\longrightarrow E_{\rm MW}(G),
 \qquad [(x_i,g_i)]\longmapsto[(x_i^2,g_i)]
\]
defines a smooth, surjective, $G$-equivariant map.  Its
set-theoretic positive square-root section is not smooth in
general.  The analogous square map to $E_{\rm LF}(G)$ is also
smooth, but is not a diffeomorphism.
\end{proposition}
\begin{proof}
On a generating chart $S_I\times G^I$, each function $x_i^2$
is smooth and the resulting finite-support coordinates define
a Magnot--Watts plot in the product-subset diffeology.  The
formula respects the relation forgetting zero-weight labels,
so finality proves smoothness.  Every barycentric weight vector
has a positive square root, proving surjectivity.  For
$|s|<1/2$ the simplex plot
\[
 s\longmapsto(1-s^2,s^2)
\]
has square root $(\sqrt{1-s^2},|s|)$, which is not differentiable
at zero.  Thus the positive square-root section is not smooth,
even for the trivial group.  Equivariance is immediate.
\end{proof}

\begin{remark}
Part~I additionally constructs a \emph{different} smooth injection
from the locally finite barycentric model into the nonnegative
spherical sector by squaring and renormalizing the barycentric
coordinates.  It is not the positive square-root section of the
original Magnot--Watts space.
\end{remark}

\subsection{Projective quotient}
The simultaneous sign action
$\varepsilon\cdot[(x_i,g_i)]=[(\varepsilon x_i,g_i)]$
commutes with the diagonal action of $G$.  We set
\[
 E^{\rm proj}(G)=E^{\rm sph}(G)/\mathbb Z_2,
 \qquad B^{\rm sph}_{\times}(G)
   =E^{\rm sph}(G)/(G\times\mathbb Z_2).
\]
For any smooth involutive automorphism
$\sigma:\mathbb Z_2\to\operatorname{Aut}(G)$, the semidirect
product $H_\sigma=G\rtimes_\sigma\mathbb Z_2$ acts instead by
\[
 (h,\varepsilon)\cdot[(x_i,g_i)]
  =[(\varepsilon x_i,h\,\sigma_\varepsilon(g_i))].
\]
The twisted quotient is denoted $B^{\rm sph}_\sigma(G)$.

\subsection{Differential forms and the scope of Part I}
Differential forms on $E^{\rm sph}(G)$ are defined plotwise.
Part~I establishes the global diffeological de Rham complex and
constructs a metric and de Rham--Dirac operator on each
\emph{finite-dimensional} fixed-support stratum. For a Fr\'echet--Lie
group, its support strata are in general infinite-dimensional; the
standard Hodge codifferential and de Rham--Dirac operator are then
not supplied without additional analytical hypotheses. A smooth
metric on the full internal
tangent pseudo-bundle additionally requires a suitable positive
kernel on $G$.  Crucially, neither a global Hodge operator nor
a global Dirac operator follows formally from the stratumwise
constructions: Part~I gives an explicit failure of automatic
gluing for Laplacians.  Sections~8--13 below will be interpreted
only with these qualifications and require a separate analytic
review.

\section{Diffeological smoothness and gluing} \label{sec:glue}
In this section we describe more precisely the basic definitions and the properties of the diffeology setteled in \cite{Magnot-MilnorI} on the spherical classifying space. 
\subsection{Final diffeology and plots}

Recall that $E^{\mathrm{sph}}(G)$ is endowed with the quotient of the final diffeology generated by the maps
\[
\psi_I : S_I \times G^I \longrightarrow \widetilde E^{\mathrm{sph}}(G),
\]
followed by the projection to $E^{\mathrm{sph}}(G)$.

\begin{proposition}[Characterization of plots]
A map $p : U \to E^{\mathrm{sph}}(G)$ is a plot if and only if for every $u \in U$, there exist:
\begin{itemize}
\item an open neighborhood $V \subset U$ of $u$,
\item a finite subset $I \subset \mathbb N$,
\item a smooth map $\widetilde p : V \to S_I \times G^I$,
\end{itemize}
such that
\[
p|_V = \psi_I \circ \widetilde p.
\]
\end{proposition}

\begin{proof}
This is the standard characterization of the final diffeology, combined with the quotient diffeology.
\end{proof}

\subsection{Local finiteness}

\begin{proposition}
Every plot $p : U \to E^{\mathrm{sph}}(G)$ is locally supported on a fixed finite index set.
\end{proposition}

\begin{proof}
By definition of the final diffeology, every point admits a neighborhood factoring through some $\psi_I$ with $I$ finite.
\end{proof}

\begin{remark}
Thus every plot uses only finitely many indices locally. When $G$ is Fr\'echet, the corresponding generating spaces need not be finite-dimensional.
\end{remark}

\subsection{Gluing of plots}

\begin{proposition}[Sheaf condition]
Let $(U_\alpha)_\alpha$ be an open cover of $U$, and let $p_\alpha : U_\alpha \to E^{\mathrm{sph}}(G)$ be plots such that
\[
p_\alpha = p_\beta \quad \text{on } U_\alpha \cap U_\beta.
\]
Then the map $p : U \to E^{\mathrm{sph}}(G)$ defined by $p|_{U_\alpha} = p_\alpha$ is a plot.
\end{proposition}

\begin{proof}
This is one of the defining axioms of a diffeology.
\end{proof}

\subsection{Strata and local models}

\begin{definition}
For each finite \(I\subset \mathbb N\), define
\[
E_I:=\psi_I(S_I\times G^I)\subset E^{\mathrm{sph}}(G).
\]
We endow \(E_I\) with the subset diffeology induced from
\(E^{\mathrm{sph}}(G)\).
\end{definition}

\begin{proposition}
For every finite \(I\subset\mathbb N\), the map
\[
\psi_I:S_I\times G^I\longrightarrow E_I
\]
is a subduction.
\end{proposition}

\begin{proof}
By definition, \(E_I\) is the image of \(\psi_I\). Hence \(\psi_I\) is
surjective.

The diffeology on \(E^{\mathrm{sph}}(G)\) is the final diffeology generated by
the maps
\[
\psi_J:S_J\times G^J\to E^{\mathrm{sph}}(G).
\]
The subset diffeology on \(E_I\) consists of those plots of
\(E^{\mathrm{sph}}(G)\) whose image lies in \(E_I\).

Since \(\psi_I\) is one of the generating maps, every plot of
\(S_I\times G^I\) is sent by \(\psi_I\) to a plot of
\(E^{\mathrm{sph}}(G)\), hence to a plot of \(E_I\).

Conversely, let $p:U\to E_I$ be a plot for the subset
 diffeology. Near any $u_0\in U$ choose a smooth lift
 $\widetilde p:V\to S_J\times G^J$ through $\psi_J$ for some
 finite $J$. Since $p(V)\subset E_I$, the intrinsically defined
 signed coordinates satisfy $x_j\equiv0$ on $V$ for every
 $j\in J\setminus I$. Dropping these identically zero
 coordinates leaves a unit spherical vector indexed by
 $J\cap I$; insert zeros and constant identity labels for
 indices $I\setminus J$. The result is a smooth map
 $V\to S_I\times G^I$ whose image under $\psi_I$ equals
 $p|_V$. Thus every subset plot locally lifts through
 $\psi_I$, so $\psi_I$ is a subduction.
\end{proof}

\begin{remark}
The map \(\psi_I\) is generally not injective. Indeed, if \(x_i=0\), then the
corresponding group coordinate \(g_i\) is irrelevant in the quotient relation.
\end{remark}

\subsection{Smoothness criterion}

\begin{proposition}
Let \(Y\) be a diffeological space. A map
\[
f:E^{\mathrm{sph}}(G)\to Y
\]
is smooth if and only if, for every finite \(I\subset\mathbb N\), the
composition
\[
f\circ \psi_I:S_I\times G^I\to Y
\]
is smooth.
\end{proposition}

\begin{proof}
Assume first that \(f\) is smooth. Since each \(\psi_I\) is smooth by
construction of the final diffeology, the composition
\[
f\circ\psi_I
\]
is smooth.

Conversely, assume that every \(f\circ\psi_I\) is smooth. Let
\[
p:U\to E^{\mathrm{sph}}(G)
\]
be a plot. Since the diffeology on \(E^{\mathrm{sph}}(G)\) is final with
respect to the family \((\psi_I)_I\), for every \(u\in U\) there exists an open
neighborhood \(V\subset U\), a finite set \(I\), and a plot
\[
\widetilde p:V\to S_I\times G^I
\]
such that
\[
p|_V=\psi_I\circ\widetilde p.
\]
Then
\[
f\circ p|_V
=
f\circ\psi_I\circ\widetilde p,
\]
which is smooth by assumption. Since smoothness of plots is local on the
domain, \(f\circ p\) is a plot of \(Y\). Hence \(f\) is smooth.
\end{proof}

\subsection{Behavior across strata}

The following point is important. The spherical model does not impose
stationarity at the hypersurfaces \(x_j=0\). Therefore a plot may cross a
boundary stratum transversally.

\begin{remark}
It is not true in general that if a coordinate \(x_j\) vanishes at a point of a
plot, then the plot locally factors through a smaller stratum. For instance,
on a sphere \(S_I\), a smooth curve may satisfy \(x_j(u_0)=0\) while
\(x_j(u)\neq0\) for \(u\neq u_0\). Thus the spherical diffeology allows
genuine transverse crossing of strata.
\end{remark}

The correct statement is the following.

\begin{proposition}
Let
\[
p:U\to E^{\mathrm{sph}}(G)
\]
be a plot. Suppose that locally
\[
p=\psi_I\circ \widetilde p
\]
with
\[
\widetilde p=(x_i,g_i)_{i\in I}:U\to S_I\times G^I.
\]
If, on an open subset \(V\subset U\), one has
\[
x_j|_V\equiv0
\]
for some \(j\in I\), then \(p|_V\) factors through the smaller model
\[
S_{I\setminus\{j\}}\times G^{I\setminus\{j\}}.
\]
\end{proposition}

\begin{proof}
Assume that \(x_j(v)=0\) for every \(v\in V\). Since
\[
\sum_{i\in I}x_i(v)^2=1,
\]
we have
\[
\sum_{i\in I\setminus\{j\}}x_i(v)^2=1
\]
for every \(v\in V\). Hence the remaining coordinates define a smooth map
\[
(x_i)_{i\in I\setminus\{j\}}:V\to S_{I\setminus\{j\}}.
\]

Similarly, by deleting the inactive group coordinate \(g_j\), we obtain a
smooth map
\[
\widetilde p_j:V\to S_{I\setminus\{j\}}\times G^{I\setminus\{j\}}.
\]

Now compare the images of \(\widetilde p|_V\) and \(\widetilde p_j\) in
\(E^{\mathrm{sph}}(G)\). Since the coordinate \(x_j\) is identically zero on
\(V\), the equivalence relation defining \(E^{\mathrm{sph}}(G)\) makes the
corresponding group coordinate \(g_j\) irrelevant. Therefore
\[
\psi_I\circ\widetilde p|_V
=
\psi_{I\setminus\{j\}}\circ\widetilde p_j.
\]
Thus \(p|_V\) factors through the smaller model.
\end{proof}

\begin{proposition}
A plot factors locally through a lower stratum precisely on those open subsets
where the corresponding spherical coordinates vanish identically.
\end{proposition}

\begin{proof}
The preceding proposition proves the sufficient condition. Conversely, if a
plot factors through \(S_{I\setminus\{j\}}\times G^{I\setminus\{j\}}\), then
the coordinate \(x_j\) is identically zero in that local representation.
\end{proof}

\begin{remark}
The spherical model allows smooth plots to cross from one
support type to another with a nonzero first-order derivative
of an inactive spherical coordinate. In barycentric coordinates
$t_j=x_j^2$, the corresponding first derivative vanishes at a
zero. This first-order distinction does not mean that the
Magnot--Watts subset diffeology requires local constancy at
the boundary.
\end{remark}
\subsection{Smooth functions}

\begin{definition}
A function
\[
f:E^{\mathrm{sph}}(G)\to\mathbb R
\]
is called smooth if for every plot
\[
p:U\to E^{\mathrm{sph}}(G),
\]
the composition
\[
f\circ p:U\to\mathbb R
\]
is smooth in the ordinary sense.
\end{definition}

\begin{proposition}
A function
\[
f:E^{\mathrm{sph}}(G)\to\mathbb R
\]
is smooth if and only if, for every finite subset
\[
I\subset\mathbb N,
\]
the composition
\[
f\circ\psi_I:S_I\times G^I\to\mathbb R
\]
is smooth.
\end{proposition}

\begin{proof}
Assume first that \(f\) is smooth. By construction of the diffeology on
\(E^{\mathrm{sph}}(G)\), each map
\[
\psi_I:S_I\times G^I\to E^{\mathrm{sph}}(G)
\]
is smooth. Therefore the composition
\[
f\circ\psi_I
\]
is smooth for every finite \(I\).

Conversely, assume that each
\[
f\circ\psi_I
\]
is smooth. Let
\[
p:U\to E^{\mathrm{sph}}(G)
\]
be a plot. Since the diffeology on \(E^{\mathrm{sph}}(G)\) is the final
diffeology generated by the family \((\psi_I)_I\), for every point
\(u\in U\) there exist:
\begin{itemize}
\item an open neighborhood \(V\subset U\) of \(u\),
\item a finite subset \(I\subset\mathbb N\),
\item and a smooth map
\[
\widetilde p:V\to S_I\times G^I
\]
such that
\[
p|_V=\psi_I\circ\widetilde p.
\]
\end{itemize}

Hence
\[
f\circ p|_V
=
f\circ\psi_I\circ\widetilde p.
\]

By assumption,
\[
f\circ\psi_I
\]
is smooth, and since \(\widetilde p\) is smooth, the composition
\[
f\circ p|_V
\]
is smooth.

Since smoothness is local on the domain, it follows that
\[
f\circ p
\]
is smooth on all of \(U\). Therefore \(f\) is smooth as a diffeological map.
\end{proof}

\subsection{Summary}

The spherical Milnor model
\[
E^{\mathrm{sph}}(G)
\]
is equipped with a diffeological structure characterized by the following
properties:
\begin{itemize}
\item every plot locally factors through a finite-index generating model
\[
S_I\times G^I;
\]
\item the diffeology is generated by the family of maps
\[
\psi_I:S_I\times G^I\to E^{\mathrm{sph}}(G);
\]
\item smooth maps and smooth functions are completely characterized by their
pull-backs through the generators \(\psi_I\);
\item the space possesses a stratified structure governed by the vanishing of
the spherical coordinates;
\item an inactive spherical coordinate may be activated at
first order; its barycentric square has vanishing first derivative
at the corresponding zero.
\end{itemize}

These properties will be used repeatedly in the sequel when constructing
geometric, topological, and analytical structures on the spherical Milnor
space.
\section{Projective quotients and \texorpdfstring{$\mathbb Z_2$}{Z2}-twisted classifying structures}
\label{sec:projective-topology}

\subsection{The antipodal action}

Recall that the spherical Milnor space is
\[
E^{\mathrm{sph}}(G)
=
\widetilde E^{\mathrm{sph}}(G)/\sim,
\]
where
\[
(x_i,g_i)\sim (x_i,g_i')
\quad\Longleftrightarrow\quad
g_i=g_i' \ \text{whenever } x_i\neq0.
\]

\begin{definition}
The antipodal action of \(\mathbb Z_2=\{\pm1\}\) on \(E^{\mathrm{sph}}(G)\) is defined by
\[
\varepsilon\cdot [(x_i,g_i)]
=
[(\varepsilon x_i,g_i)].
\]
\end{definition}

\begin{proposition}
The antipodal action of \(\mathbb Z_2\) on \(E^{\mathrm{sph}}(G)\) is well-defined, smooth and free.
\end{proposition}

\begin{proof}
We first prove that the action is well-defined. Suppose that
\[
(x_i,g_i)\sim (x_i,g_i').
\]
Then \(g_i=g_i'\) for every \(i\) such that \(x_i\neq0\). Since
\[
\varepsilon x_i\neq0
\quad\Longleftrightarrow\quad
x_i\neq0,
\]
we also have \(g_i=g_i'\) for every \(i\) such that \(\varepsilon x_i\neq0\). Hence
\[
(\varepsilon x_i,g_i)\sim(\varepsilon x_i,g_i'),
\]
so the action is well-defined on the quotient.

Smoothness follows from the definition of the final diffeology. On each finite model
\[
S_I\times G^I,
\]
the action is given by
\[
(x_i,g_i)_{i\in I}
\longmapsto
(\varepsilon x_i,g_i)_{i\in I},
\]
which is smooth since \(S_I\subset \mathbb R^I\) carries the subset diffeology.

Finally, suppose that \(-1\) fixes a point \([(x_i,g_i)]\). Then
\[
[(x_i,g_i)]=[(-x_i,g_i)].
\]
The equivalence relation does not change the spherical coordinates, hence this implies
\[
x_i=-x_i
\]
for every \(i\). Therefore \(x_i=0\) for all \(i\), contradicting
\[
\sum_i x_i^2=1.
\]
Thus the action is free.
\end{proof}

\begin{definition}
The projective spherical Milnor space is
\[
E^{\mathrm{proj}}(G)
:=
E^{\mathrm{sph}}(G)/\mathbb Z_2,
\]
equipped with the quotient diffeology.
\end{definition}

\subsection{Contractibility by collision-free reindexing}
\label{subsec:II-reindex}
The naive linear interpolation between a labelled vector and its
unilateral shift is not legitimate: $x$ and its shift need not be
orthogonal, and, more importantly, entries with different group
labels would be superposed.  We use the collision-free homotopy
of Part~I instead.

\begin{lemma}[Equivariant reindexing]
\label{lem:II-reindex}
Let $\tau:\mathbb N\to\mathbb N$ be strictly increasing and
$\tau(n)\ge n$.  Replacing the coordinate $(x_n,g_n)$ by
$(x_n,g_n)$ at index $\tau(n)$ and by zero at the other indices
defines a smooth map $R_\tau:E^{\rm sph}(G)\to E^{\rm sph}(G)$.
It is smoothly homotopic to the identity, equivariantly for
both $G\times\mathbb Z_2$ and every $G\rtimes_\sigma\mathbb Z_2$
action above.
\end{lemma}
\begin{proof}
Partition $(0,1]$ into $I_n=[1/(n+1),1/n]$ and choose
$b:[0,1]\to[0,1]$ smooth, equal to $0$ near $0$ and to $1$
near $1$.  At time $t\in I_n$, move coordinate $n$ to
$\tau(n)$ using the two signed weights
\[
 x_n\cos\!\left(\tfrac\pi2 b(u_n(t))\right),
 \qquad x_n\sin\!\left(\tfrac\pi2 b(u_n(t))\right),
 \quad u_n(t)=\frac{t-1/(n+1)}{1/n-1/(n+1)}.
\]
Attach the \emph{same} label $g_n$ to both coordinates.
If $\tau(n)=n$, leave this coordinate fixed. Otherwise
indices $k>n$ have already moved to $\tau(k)>\tau(n)$,
while $k<n$ still occupy positions below $n$, so the target
position is unoccupied.  There are no collisions of distinct
labels.  At $t=0$ use the identity map.  On a finite-support
chart with indices bounded by $N$, nothing changes when
$t<1/(N+1)$.  This proves smoothness at $t=0$; the flat
cut-off $b$ gives smooth joins at endpoints of the intervals.
The norm is preserved by $\cos^2+\sin^2=1$ and the output
support remains in the finite union $I\cup\tau(I)$ on each
input chart $I$.  The construction is therefore smooth for
the final diffeology.  It merely reindexes labels and multiplies
signed coordinates by scalars, so it commutes with the global
sign and the simultaneous $G$-multiplication or $\sigma$-action.
\end{proof}

\begin{theorem}[Contractibility and equivariant joining]
\label{thm:II-contract}
For any diffeological group $G$, $E^{\rm sph}(G)$ is smoothly
contractible.  Moreover, for $H=G\rtimes_\sigma\mathbb Z_2$,
any two smooth $H$-equivariant maps from an $H$-space $P$ to
$E^{\rm sph}(G)$ are smoothly $H$-equivariantly homotopic.
\end{theorem}
\begin{proof}
Apply Lemma~\ref{lem:II-reindex} to $\tau(n)=n+1$.  At the
endpoint the first spherical coordinate is zero.  On the image
of this shift, interpolate to the point supported at index one
with label $e_G$ by putting $\sin(\pi t/2)$ at index one and
multiplying all shifted coordinates by $\cos(\pi t/2)$.
This is smooth and never superposes different labels.  A smooth
reparametrization constant near the concatenation time gives a
contraction.  It is not claimed to be equivariant.

For the second assertion take $\tau_e(n)=2n$ and
$\tau_o(n)=2n-1$.  Equivariantly reindex the two given maps
to disjoint even and odd positions.  Interpolate between their
spherical coordinates with factors $\cos(\pi t/2)$ and
$\sin(\pi t/2)$, respectively, retaining all group labels.
The squared norm is one, supports are disjoint and locally
finite, and the construction is smooth on local generating
charts.  The $H$-action transforms all labels by the same
$\sigma$-automorphism and simultaneously multiplies all
coordinates by the sign, so the homotopy is $H$-equivariant.
Concatenating with the two reindexing homotopies proves the claim.
\end{proof}

\subsection{The projective bundle and its degree-one class}
\begin{proposition}
\label{prop:II-projective}
The sign quotient $q:E^{\rm sph}(G)\to E^{\rm proj}(G)$
is a smooth principal $\mathbb Z_2$-bundle.  The local sections constructed below give a canonical
$\mathbb Z_2$-valued Cech class on the $D$-open cover of
$E^{\rm proj}(G)$. An identification with singular
$H^1(E^{\rm proj}(G);\mathbb Z_2)$ is asserted only
when the necessary comparison hypotheses on the $D$-topology
hold.
\end{proposition}
\begin{proof}
At every point some $x_i\ne0$.  On the sign-invariant open
set where $x_i\ne0$, choose the unique representative with
$x_i>0$.  Multiplication by $\operatorname{sgn}(x_i)$ is a
smooth locally constant operation on this open set and is
invariant under the sign action.  It therefore defines a
smooth local section of the quotient; the inverse to the
local trivialization is obtained by recording the sign.
These sets cover the projective quotient.  The transition
functions are locally constant, satisfy the Cech cocycle
condition, and represent the class of the double cover.
\end{proof}

\begin{remark}
Since the total space is smoothly contractible, this quotient
plays a classifying role for $\mathbb Z_2$ on the class of
bases for which the smooth principal-bundle classification
argument of Section~5 applies.  Contractibility alone is not
used as a substitute for that argument.
\end{remark}

\subsection{The \(G\)-action and the product quotient}

The group \(G\) acts on \(E^{\mathrm{sph}}(G)\) by
\[
h\cdot [(x_i,g_i)]
=
[(x_i,hg_i)].
\]

\begin{proposition}
The actions of \(G\) and \(\mathbb Z_2\) on \(E^{\mathrm{sph}}(G)\) commute.
\end{proposition}

\begin{proof}
For \(h\in G\) and \(\varepsilon\in\mathbb Z_2\),
\[
h\cdot(\varepsilon\cdot[(x_i,g_i)])
=
h\cdot[(\varepsilon x_i,g_i)]
=
[(\varepsilon x_i,hg_i)].
\]
On the other hand,
\[
\varepsilon\cdot(h\cdot[(x_i,g_i)])
=
\varepsilon\cdot[(x_i,hg_i)]
=
[(\varepsilon x_i,hg_i)].
\]
The two expressions agree.
\end{proof}

\begin{definition}
The product spherical classifying quotient is
\[
B^{\mathrm{sph}}_{\times}(G)
:=
E^{\mathrm{sph}}(G)/(G\times \mathbb Z_2).
\]
\end{definition}

\begin{proposition}
There are canonical diffeomorphisms
\[
(E^{\mathrm{sph}}(G)/G)/\mathbb Z_2
\cong
(E^{\mathrm{sph}}(G)/\mathbb Z_2)/G
\cong
B^{\mathrm{sph}}_{\times}(G).
\]
\end{proposition}

\begin{proof}
Since the two actions commute, they define a smooth action of the product group
\[
G\times\mathbb Z_2.
\]
The orbit of a point under first quotienting by \(G\) and then by \(\mathbb Z_2\) is exactly its orbit under \(G\times\mathbb Z_2\). The same is true if the order of the quotients is reversed.

The set-theoretic identifications are therefore clear. The diffeological identifications follow from the universal property of quotient diffeologies: a map out of any of these quotients is smooth if and only if its pullback to \(E^{\mathrm{sph}}(G)\) is smooth and invariant under both actions. This condition is independent of the order in which the two quotients are taken.
\end{proof}

\subsection{Twisted quotients and semidirect products}

The preceding construction corresponds to the product group \(G\times\mathbb Z_2\). To obtain genuinely twisted \(G\)-bundles, one must introduce an action of \(\mathbb Z_2\) on \(G\).

Let
\[
\sigma:\mathbb Z_2\longrightarrow \operatorname{Aut}(G)
\]
be a smooth action. Write
\[
G\rtimes_\sigma \mathbb Z_2
\]
for the corresponding semidirect product.

\begin{definition}
The \(\sigma\)-twisted action of \(G\rtimes_\sigma\mathbb Z_2\) on \(E^{\mathrm{sph}}(G)\) is defined by
\[
(h,\varepsilon)\cdot[(x_i,g_i)]
=
[(\varepsilon x_i, h\,\sigma^\varepsilon(g_i))].
\]
\end{definition}

\begin{proposition}
This defines a smooth action of \(G\rtimes_\sigma\mathbb Z_2\) on \(E^{\mathrm{sph}}(G)\).
\end{proposition}

\begin{proof}
We first check that the formula is compatible with the quotient relation. If
\[
(x_i,g_i)\sim(x_i,g_i'),
\]
then \(g_i=g_i'\) whenever \(x_i\neq0\). Since \(\varepsilon x_i\neq0\) if and only if \(x_i\neq0\), we get
\[
h\,\sigma^\varepsilon(g_i)
=
h\,\sigma^\varepsilon(g_i')
\]
for all active indices. Hence the formula is well-defined.

We now verify the group law. The multiplication in \(G\rtimes_\sigma\mathbb Z_2\) is
\[
(h,\varepsilon)(k,\eta)
=
(h\,\sigma^\varepsilon(k),\varepsilon\eta).
\]
Then
\[
(k,\eta)\cdot[(x_i,g_i)]
=
[(\eta x_i,k\,\sigma^\eta(g_i))].
\]
Applying \((h,\varepsilon)\) gives
\[
(h,\varepsilon)\cdot[(\eta x_i,k\,\sigma^\eta(g_i))]
=
[(\varepsilon\eta x_i,
h\,\sigma^\varepsilon(k\,\sigma^\eta(g_i)))].
\]
Since \(\sigma^\varepsilon\) is an automorphism,
\[
\sigma^\varepsilon(k\,\sigma^\eta(g_i))
=
\sigma^\varepsilon(k)\,\sigma^{\varepsilon\eta}(g_i).
\]
Thus
\[
(h,\varepsilon)\cdot((k,\eta)\cdot[(x_i,g_i)])
=
[(\varepsilon\eta x_i,
h\,\sigma^\varepsilon(k)\,\sigma^{\varepsilon\eta}(g_i))],
\]
which is precisely
\[
(h\,\sigma^\varepsilon(k),\varepsilon\eta)\cdot[(x_i,g_i)].
\]
Hence this is an action.

Smoothness follows from smoothness of the \(G\)-action, of the finite \(\mathbb Z_2\)-action, and of the automorphism \(\sigma^\varepsilon\) on each finite chart \(S_I\times G^I\).
\end{proof}

\begin{definition}
The \(\sigma\)-twisted spherical classifying quotient is
\[
B^{\mathrm{sph}}_\sigma(G)
:=
E^{\mathrm{sph}}(G)/(G\rtimes_\sigma\mathbb Z_2).
\]
\end{definition}

\begin{remark}
When \(\sigma\) is trivial, \(G\rtimes_\sigma\mathbb Z_2=G\times\mathbb Z_2\), and \(B^{\mathrm{sph}}_\sigma(G)=B^{\mathrm{sph}}_{\times}(G)\).
\end{remark}

\begin{remark}
For nontrivial \(\sigma\), the \(G\)- and \(\mathbb Z_2\)-parts no longer commute as independent actions. Thus the twisted quotient must be understood as a quotient by the semidirect product, not as an iterated quotient by commuting actions.
\end{remark}

\subsection{Cocycle interpretation}

Let \(X\) be a diffeological space and let
\[
P\longrightarrow X
\]
be a principal \(G\rtimes_\sigma\mathbb Z_2\)-bundle. Choose a cover \((U_i)\) and transition functions
\[
h_{ij}:U_{ij}\longrightarrow G\rtimes_\sigma\mathbb Z_2.
\]
Write
\[
h_{ij}=(g_{ij},\varepsilon_{ij}).
\]
The cocycle condition
\[
h_{ij}h_{jk}=h_{ik}
\]
becomes
\[
\varepsilon_{ij}\varepsilon_{jk}=\varepsilon_{ik},
\]
and
\[
g_{ij}\,\sigma^{\varepsilon_{ij}}(g_{jk})=g_{ik}.
\]

\begin{proposition}
The \(\mathbb Z_2\)-components \((\varepsilon_{ij})\) define a class
\[
\alpha(P)\in H^1(X;\mathbb Z_2).
\]
The \(G\)-components \((g_{ij})\) form a \(G\)-valued cocycle twisted by \(\alpha(P)\).
\end{proposition}

\begin{proof}
The equation
\[
\varepsilon_{ij}\varepsilon_{jk}=\varepsilon_{ik}
\]
is exactly the Čech \(1\)-cocycle condition for \(\mathbb Z_2\)-valued transition functions. Hence it defines a class in \(H^1(X;\mathbb Z_2)\).

The second equation
\[
g_{ij}\,\sigma^{\varepsilon_{ij}}(g_{jk})=g_{ik}
\]
says precisely that the \(G\)-valued transition functions fail to satisfy the ordinary cocycle condition by the twisting induced by the \(\mathbb Z_2\)-cocycle. This is the usual cocycle description of a \(G\)-bundle twisted by the local system determined by \(\alpha(P)\).
\end{proof}

\begin{remark}
Thus the topological information carried by the projective spherical quotient is not merely the product of a \(G\)-bundle and a double cover. In the presence of a nontrivial action \(\sigma\), it encodes a genuinely twisted principal \(G\)-bundle.
\end{remark}

\section{Classifying properties, characteristic cocycles and canonical forms}
\label{sec:classifying-properties}

\subsection{The actual principal semidirect-product bundle}
\label{subsec:II-universality}
Let $\sigma:\mathbb Z_2\to\operatorname{Aut}(G)$ be a smooth
action and $H_\sigma=G\rtimes_\sigma\mathbb Z_2$.
The left action on $E=E^{\rm sph}(G)$ is
\[
 (h,\varepsilon)\cdot[(x_i,g_i)]
 =[(\varepsilon x_i,h\sigma_\varepsilon(g_i))].
\]

\begin{proposition}[Principal-bundle structure]
\label{prop:II-H-principal}
The $H_\sigma$-action is smooth and free.  Its quotient
$E\to B^{\rm sph}_\sigma(G)$ is a diffeological principal
$H_\sigma$-bundle admitting smooth local trivializations on
the invariant open sets where a fixed $x_i$ is nonzero.
\end{proposition}
\begin{proof}
Smoothness follows on finite generating charts.  If
$(h,\varepsilon)$ fixes a point, some nonzero $x_i$ implies
$\varepsilon=1$, and its active group label then gives $h=e$.
For local triviality, set $V_i=\{z\in E:x_i(z)\ne0\}$, an
invariant $D$-open set.  On $V_i$ define the unique gauge
normalization sending the $i$-th weight to a positive number
and the $i$-th group label to $e_G$:
\[
 N_i(z)=(\sigma_{\varepsilon_i(z)}(g_i(z))^{-1},
               \varepsilon_i(z))\cdot z,
 \qquad \varepsilon_i(z)=\operatorname{sgn}(x_i(z)).
\]
This is smooth on local generating lifts, since $x_i$ is
nonzero there, and constant on $H_\sigma$-orbits.  Hence
it descends to a smooth section over $V_i/H_\sigma$.
The original point is recovered uniquely from its normalized
representative and the group element used to normalize it;
both depend smoothly on the point because the $i$-th label
and sign are smooth on $V_i$.  This gives the claimed
trivialization and establishes the principal-bundle condition
without inferring smoothness of the division map merely from
freeness.
\end{proof}

\begin{proposition}[$D$-numerability of the spherical $H_\sigma$-bundle]
\label{prop:II-H-numerable}
For every diffeological group $G$ and smooth involution
$\sigma$, the principal bundle
$E^{\rm sph}(G)\to B^{\rm sph}_\sigma(G)$ is $D$-numerable.
\end{proposition}
\begin{proof}
The functions $\lambda_i=x_i^2$ descend to $B$ and
satisfy $\sum_i\lambda_i=1$.  Their supports are pointwise
finite, but need not be locally finite.  Fix $\delta_i>0$
with $\sum_i\delta_i<3/4$, and a smooth nonnegative
$\chi$ vanishing on $(-\infty,0]$ and positive on $(0,\infty)$.
Set
\[
 \rho_i=\chi(\lambda_i-\delta_i)
 \chi\Bigl(\frac34-\sum_{j<i}\lambda_j\Bigr).
\]
The supports of $\rho_i$ lie in $\{\lambda_i>0\}$.
At each point, a partial sum $\sum_{j\le N}\lambda_j$
exceeds $3/4$, so all $\rho_i$ with $i>N$ vanish on
an entire neighborhood; hence the family is locally finite.
Its sum is everywhere positive: at the first index whose
partial sum reaches $3/4$, vanishing of all earlier $\rho_i$
would force every earlier $\lambda_i\le\delta_i$, contrary
to $\sum_i\delta_i<3/4$.  Thus
$\rho_i/\sum_j\rho_j$ is a smooth locally finite
partition of unity subordinate to the canonical
trivializing cover.
\end{proof}

\begin{lemma}[A smooth classifying map without square roots]
\label{lem:II-H-map}
Let $P\to X$ be a principal \emph{left} $H_\sigma$-bundle
with countable $D$-locally finite smooth trivializing cover
$(U_i)$ and smooth functions $\lambda_i\ge0$ satisfying
$\sum_i\lambda_i=1$ and $\operatorname{supp}\lambda_i\subset U_i$.
Then there is a smooth $H_\sigma$-equivariant map
$F:P\to E^{\rm sph}(G)$.
\end{lemma}
\begin{proof}
Choose smooth local sections $s_i:U_i\to P$ and put
\[
 q(x)=\left(\sum_j\lambda_j(x)^2\right)^{1/2}>0,
 \qquad \rho_i(x)=\lambda_i(x)/q(x).
\]
The sums are locally finite, each $\rho_i$ is smooth and
$\sum_i\rho_i^2=1$.  Whenever $x=\pi(p)\in U_i$ write
uniquely $p=(g_i(p),\varepsilon_i(p))\cdot s_i(x)$.
Assign the $i$-th spherical coordinate and group label
\[
 X_i(p)=\varepsilon_i(p)\rho_i(\pi p),
 \qquad L_i(p)=g_i(p),
\]
where $\rho_i\ne0$; if $\rho_i=0$ the label is forgotten.
For an element $(h,\eta)\in H_\sigma$ the local coordinates
satisfy
\[
 \varepsilon_i((h,\eta)p)=\eta\varepsilon_i(p),
 \quad g_i((h,\eta)p)=h\sigma_\eta(g_i(p)),
\]
so $F((h,\eta)p)=(h,\eta)F(p)$.

To verify smoothness, fix $p_0\in P$ and a plot through it.
Local finiteness yields a neighborhood on which only finitely
many supports of $\lambda_i$ occur.  Whenever
$\pi(p_0)\notin U_i$, the support condition permits us to
shrink that neighborhood so $\rho_i$ vanishes identically.
All remaining $g_i,\varepsilon_i$ are smooth on the
neighborhood, giving an actual lift to $S_I\times G^I$.
Thus $F$ is a diffeological smooth map, not merely a
coordinatewise-smooth assignment.
\end{proof}

\begin{theorem}[Classification without a regularity hypothesis]
\label{thm:II-H-classification}
Let $G$ be any diffeological group and
$\sigma:\mathbb Z_2\to\operatorname{Aut}(G)$ a smooth action.
For every diffeological base $X$, pullback induces a
natural bijection
\[
 [X,B^{\rm sph}_\sigma(G)]_{\rm sm}
 \ \cong\ \operatorname{Prin}^{D\text{-}\rm num}_{H_\sigma}(X).
\]
No regularity hypothesis or connection is used.
\end{theorem}
\begin{proof}
By Proposition~\ref{prop:II-H-numerable}, pullbacks of
the universal bundle are $D$-numerable.  Conversely,
by the countable numerable refinement of
\cite[Section~4]{ChristensenWu2021Classifying}, a
$D$-numerable bundle has the locally finite data required
by Lemma~\ref{lem:II-H-map}.  The equivariant map produced
there induces a smooth classifying map and an isomorphism
with the pullback of the spherical universal bundle.
The collision-free equivariant contraction from
Theorem~\ref{thm:II-contract} gives uniqueness of
classifying maps up to smooth homotopy among choices of
partitions and local sections.
Finally, homotopy invariance of pullbacks of $D$-numerable
principal bundles holds for arbitrary diffeological groups,
by \cite[Corollary~4.14]{ChristensenWu2021Classifying}:
it follows from smoothly locally finite trivializations
and a cylinder gluing argument, not horizontal transport.
The two assignments are inverse and natural.
\end{proof}

\subsection{Explicit construction of classifying maps}
For a $D$-numerable principal bundle as in Theorem~\ref{thm:II-H-classification},
the classifying map is the orbit map of the explicit
$H_\sigma$-equivariant construction of
Lemma~\ref{lem:II-H-map}.  In particular, the spherical
weights are
\[
 X_i=\varepsilon_i\frac{\lambda_i}{\sqrt{\sum_j\lambda_j^2}},
\]
not $\sqrt{\lambda_i}$.  The active group labels are
$g_i\in G$, not the semidirect-product elements
$(g_i,\varepsilon_i)\in H_\sigma$.  The signs are encoded
in the spherical coordinates.  This is the mechanism
relating the product-semidirect action to the projective
quotient, and it is indispensable for a correct proof of
smoothness.

\subsection{The associated \texorpdfstring{\(\mathbb Z_2\)}{Z2}-class}

The projection
\[
H_\sigma=G\rtimes_\sigma\mathbb Z_2
\longrightarrow
\mathbb Z_2
\]
induces a map of classifying spaces
\[
B^{\mathrm{sph}}_\sigma(G)
\longrightarrow
B\mathbb Z_2.
\]

\begin{definition}
Let
\[
\alpha_{\mathrm{univ}}\in H^1(B^{\mathrm{sph}}_\sigma(G);\mathbb Z_2)
\]
be the pull-back of the universal class in
\[
H^1(B\mathbb Z_2;\mathbb Z_2).
\]
\end{definition}

\begin{proposition}
For a principal \(H_\sigma\)-bundle \(P\to X\) classified by \(f_P\), the class
\[
f_P^*\alpha_{\mathrm{univ}}\in H^1(X;\mathbb Z_2)
\]
is represented by the Čech cocycle \((\varepsilon_{ij})\).
\end{proposition}

\begin{proof}
The \(\mathbb Z_2\)-bundle associated with \(P\) is obtained by quotienting out the \(G\)-part:
\[
P/G\longrightarrow X.
\]
Its transition functions are precisely the \(\mathbb Z_2\)-components \(\varepsilon_{ij}\) of the \(H_\sigma\)-transition functions
\[
h_{ij}=(g_{ij},\varepsilon_{ij}).
\]
Thus \((\varepsilon_{ij})\) represents the first cohomology class associated with the induced double cover. Since \(f_P\) classifies \(P\), this class is the pull-back of the universal class.
\end{proof}

\begin{remark}
Thus the class \(f_P^*\alpha_{\mathrm{univ}}\) is the topological obstruction measuring the failure of the \(H_\sigma\)-bundle to reduce to an ordinary \(G\)-bundle.
\end{remark}

\subsection{Twisted cocycles}

The cocycle condition for \(h_{ij}\) reads
\[
h_{ij}h_{jk}=h_{ik}.
\]
Writing
\[
h_{ij}=(g_{ij},\varepsilon_{ij}),
\]
one obtains
\[
\varepsilon_{ij}\varepsilon_{jk}=\varepsilon_{ik},
\]
and
\[
g_{ij}\sigma^{\varepsilon_{ij}}(g_{jk})=g_{ik}.
\]

\begin{proposition}
Let
\[
H_\sigma=G\rtimes_\sigma \mathbb Z_2
\]
with product
\[
(g,\varepsilon)(g',\varepsilon')
=
(g\,\sigma_\varepsilon(g'),\varepsilon\varepsilon').
\]
Let \(P\to X\) be a principal \(H_\sigma\)-bundle, trivialized over an open
cover \((U_i)_i\). Write its transition functions as
\[
h_{ij}=(g_{ij},\varepsilon_{ij})
:
U_{ij}\to G\rtimes_\sigma\mathbb Z_2.
\]
Then the functions \(\varepsilon_{ij}\) define a \(\mathbb Z_2\)-valued
Čech cocycle, hence a class
\[
\alpha(P)=[\varepsilon_{ij}]\in H^1(X;\mathbb Z_2),
\]
and the functions \(g_{ij}\) satisfy the twisted cocycle relation
\[
g_{ij}\,\sigma_{\varepsilon_{ij}}(g_{jk})
=
g_{ik}
\qquad
\text{on }U_{ijk}.
\]
Thus the pair
\[
(\varepsilon_{ij},g_{ij})
\]
defines a non-abelian \(G\)-valued cocycle twisted by the local system
associated with \(\alpha(P)\).
\end{proposition}

\begin{proof}
Since \(P\to X\) is a principal \(H_\sigma\)-bundle, its transition functions
satisfy the usual non-abelian Čech cocycle condition
\[
h_{ij}h_{jk}=h_{ik}
\qquad
\text{on }U_{ijk}.
\]
Writing
\[
h_{ij}=(g_{ij},\varepsilon_{ij}),
\qquad
h_{jk}=(g_{jk},\varepsilon_{jk}),
\]
and using the product in the semidirect product, we obtain
\[
h_{ij}h_{jk}
=
(g_{ij},\varepsilon_{ij})(g_{jk},\varepsilon_{jk})
=
(g_{ij}\sigma_{\varepsilon_{ij}}(g_{jk}),
\varepsilon_{ij}\varepsilon_{jk}).
\]
Therefore the equality
\[
h_{ij}h_{jk}=h_{ik}
=
(g_{ik},\varepsilon_{ik})
\]
is equivalent to the two equations
\[
\varepsilon_{ij}\varepsilon_{jk}
=
\varepsilon_{ik},
\]
and
\[
g_{ij}\sigma_{\varepsilon_{ij}}(g_{jk})
=
g_{ik}.
\]

The first equation is exactly the Čech cocycle condition for the
\(\mathbb Z_2\)-valued transition functions \((\varepsilon_{ij})\). Hence it
defines a class
\[
\alpha(P)\in H^1(X;\mathbb Z_2).
\]

The second equation says that the \(G\)-valued transition functions do not
satisfy the ordinary cocycle condition unless the action of \(\mathbb Z_2\)
is trivial. Instead, the term \(g_{jk}\) is transported by the automorphism
\[
\sigma_{\varepsilon_{ij}}\in \operatorname{Aut}(G).
\]
Thus the \(G\)-valued part is a non-abelian cocycle with coefficients twisted
by the \(\mathbb Z_2\)-local system represented by
\[
(\varepsilon_{ij}).
\]
This proves the claim.
\end{proof}
\begin{remark}
If \(\sigma\) is trivial, the relation reduces to the ordinary non-abelian
cocycle condition
\[
g_{ij}g_{jk}=g_{ik}.
\]
Thus the above construction is exactly the usual Čech description of a
principal \(G\)-bundle, modified by the \(\mathbb Z_2\)-twisting.
\end{remark}
\subsection{Pull-back of the barycentric connection}
\label{subsec:II-pullback-conn}
The left $H_\sigma$-principal bundle $E\to B^{\rm sph}_\sigma(G)$
has the $\mathfrak g$-valued connection
\[
 \Theta=\sum_i x_i^2\,dg_i\,g_i^{-1},
\]
where the right Maurer--Cartan form matches the left-action
convention of Part~I.  The discrete subgroup acts on this
connection through $d\sigma$.

\begin{proposition}
Let $F:P\to E$ be the equivariant classifying map of
Lemma~\ref{lem:II-H-map}.  Then $\omega_P=F^*\Theta$
is a principal connection on $P$.  If
$s_k:U_k\to P$ is a local section, write
$s_k(x)=(g_i(s_k(x)),\varepsilon_i(s_k(x)))\cdot s_i(x)$
where $x\in U_i\cap U_k$.  The local connection form is
\[
 A_k=s_k^*\omega_P
  =\sum_i\rho_i^2\,d(g_i\circ s_k)
                     (g_i\circ s_k)^{-1},
 \qquad \rho_i=\frac{\lambda_i}{\sqrt{\sum_j\lambda_j^2}}.
\]
The terms vanish off their support and this formula is
well-defined locally as a finite sum.
\end{proposition}
\begin{proof}
The connection is equivariant and reproduces infinitesimal
$\mathfrak g$-fundamental vectors.  Equivariance of $F$
therefore makes $F^*\Theta$ a connection on $P$.  In a
local section, the squared signed coordinate is $\rho_i^2$
and the active $G$-label is $g_i\circ s_k$.  Pulling back
$\Theta$ yields the stated formula, which is smooth because
each $\rho_i$ has support in $U_i$ and the sum is
locally finite.  Notice that a connection form on the total
space $E$ is pulled back by $F:P\to E$, \emph{not} by the
base map $f_P:X\to B^{\rm sph}_\sigma(G)$.
\end{proof}

\subsection{Curvature and characteristic forms}
For a left-action principal bundle, with the right
Maurer--Cartan convention above, define curvature by
\[
 F_k=dA_k-\tfrac12[A_k\wedge A_k].
\]
If local sections satisfy $s_\ell=h_{\ell k}\cdot s_k$,
then
\[
 A_\ell=\operatorname{Ad}_{h_{\ell k}}A_k
                  +dh_{\ell k}\,h_{\ell k}^{-1},
 \qquad F_\ell=\operatorname{Ad}_{h_{\ell k}}F_k.
\]
Here $\operatorname{Ad}_{(h,\eta)}=
\operatorname{Ad}_h\circ d\sigma_\eta$ acts on
$\mathfrak g$; the displayed inhomogeneous term is
interpreted in the identity component since $\eta$ is
locally constant.

\begin{theorem}[Conditional Chern--Weil forms]
If $P:\mathfrak g^r\to\mathbb R$ is a continuous symmetric
$r$-linear form invariant under $\operatorname{Ad}(H_\sigma)$,
then $P(F_k,\dots,F_k)$ is a well-defined closed global
$2r$-form on $X$.  Its diffeological de Rham class is
independent of the connection among connections joined by
smooth paths.
\end{theorem}
\begin{proof}
The transformation law and $H_\sigma$-invariance of $P$
show that the local forms glue.  The Bianchi identity in
our convention reads
$dF-[A\wedge F]=0$.  Invariance of $P$ cancels all
commutator terms in the exterior differential of $P(F^r)$,
so the resulting form is closed.  For a smooth path $A_t$,
differentiating the invariant polynomial and integrating
the usual Chern--Simons transgression in $t$ gives an exact
difference.  Continuity of $P$ is needed for smooth
infinite-dimensional groups.
\end{proof}

\subsection{Relation with known topological invariants}

The projection
\[
H_\sigma\to \mathbb Z_2
\]
gives the degree-one invariant
\[
\alpha(P)\in H^1(X;\mathbb Z_2).
\]

The inclusion of \(G\) into \(H_\sigma\) and the invariant polynomials on \(\mathfrak g\) yield the usual characteristic classes of the underlying \(G\)-geometry, whenever they are invariant under the \(\mathbb Z_2\)-action.

\begin{theorem}
Let \(P\to X\) be a principal \(H_\sigma\)-bundle classified by
\[
f_P:X\to B^{\mathrm{sph}}_\sigma(G).
\]
Then:
\begin{enumerate}
\item the class \(f_P^*\alpha_{\mathrm{univ}}\) is the first \(\mathbb Z_2\)-obstruction class of \(P\);
\item the characteristic forms obtained by pulling back the universal barycentric connection coincide with the Chern--Weil forms of the associated \(H_\sigma\)-connection;
\item when \(\sigma\) is trivial, these invariants split into the ordinary characteristic classes of the \(G\)-bundle together with the independent class in \(H^1(X;\mathbb Z_2)\).
\end{enumerate}
\end{theorem}

\begin{proof}
The first point follows from the preceding identification of \(f_P^*\alpha_{\mathrm{univ}}\) with the Čech class of the \(\mathbb Z_2\)-components \((\varepsilon_{ij})\).

The second point follows from the explicit pull-back formula
\[
 A_k=\sum_i\rho_i^2\,d(g_i\circ s_k)(g_i\circ s_k)^{-1},
\]
which is precisely the local expression for the induced connection on \(P\). The characteristic forms are therefore the usual Chern--Weil forms of this connection.

If \(\sigma\) is trivial, then
\[
H_\sigma=G\times \mathbb Z_2.
\]
Thus every principal \(H_\sigma\)-bundle is equivalent to a pair consisting of a principal \(G\)-bundle and a principal \(\mathbb Z_2\)-bundle. The corresponding invariants split accordingly.
\end{proof}

\section{Spherical lifting cocycles and Dixmier--Douady obstructions}
\label{sec:gerbes}

The geometric input in this section is the \emph{explicit family of
local sections} of the spherical quotient, not a preassigned
cohomology class.  We first compute its transition functions and
then measure the failure of their lifts to satisfy the cocycle
condition.  This yields an \emph{obstruction to lifting a specified
principal bundle through a specified extension}; no classification
of all higher twists is asserted.

\subsection{Transitions obtained directly from spherical coordinates}
\label{subsec:II-spherical-transition}
Set $E=E^{\mathrm{sph}}(G)$ and $B=E/H_\sigma$, where
$H_\sigma=G\rtimes_\sigma\mathbb Z_2$ acts on the left by
\[
 (h,\eta)\cdot[(x_k,g_k)]=[(\eta x_k,h\sigma_\eta(g_k))].
\]
The functions $\lambda_i([x,g])=x_i^2$ are invariant, so
\[
 U_i=\{b\in B:\lambda_i(b)>0\}
\]
are $D$-open sets.  Over $U_i$ the nonzero coordinate $x_i$
and its label $g_i$ determine a canonical normalization.
For a representative $z=[(x_k,g_k)]$ over $b\in U_i$ put
\[
 \eta_i(z)=\operatorname{sgn}(x_i),\qquad
 \kappa_i(z)=\bigl(\sigma_{\eta_i(z)}(g_i)^{-1},\eta_i(z)\bigr)
 \in H_\sigma .
\]
The point $s_i(b)=\kappa_i(z)\cdot z$ is independent of the
representative $z$ in the $H_\sigma$-orbit: it is uniquely
characterized by having $x_i>0$ and $g_i=e_G$.
Consequently $s_i:U_i\to E$ is a smooth section.

On $U_{ij}$ use the left-action convention
\begin{equation}\label{eq:II-spherical-hij-def}
 s_i=h_{ij}\cdot s_j,
 \qquad h_{ij}=\kappa_i(z)\kappa_j(z)^{-1}.
\end{equation}
Writing $h_{ij}=(g_{ij},\varepsilon_{ij})$, multiplication
in the semidirect product gives the \emph{explicit formulas}
\begin{equation}\label{eq:II-spherical-hij}
 \boxed{\quad
 \varepsilon_{ij}=\eta_i\eta_j,\qquad
 g_{ij}=\sigma_{\eta_i}(g_i^{-1}g_j).
 \quad}
\end{equation}
Here the expressions on the right are evaluated on any
representative $z$ over $b$; their combination is invariant
under replacement of $z$ by $(h,\eta)\cdot z$.
Associativity of the normalizations yields
\begin{equation}\label{eq:II-spherical-cocycle}
 \varepsilon_{ij}\varepsilon_{jk}=\varepsilon_{ik},\qquad
 g_{ij}\sigma_{\varepsilon_{ij}}(g_{jk})=g_{ik}.
\end{equation}
\begin{proposition}[Spherical origin and smoothness of the transitions]
\label{prop:II-spherical-local-transitions}
The above maps $s_i$ are well-defined smooth local sections
of $E\to B$.  Their transition maps are precisely
\eqref{eq:II-spherical-hij}; these are smooth, satisfy
\eqref{eq:II-spherical-cocycle}, and depend only on the
spherical quotient, not on the extension $\widehat H$.
\end{proposition}
\begin{proof}
On the $D$-open subset of $E$ where $x_i\ne0$,
the sign $\eta_i$ is locally constant and the active label
$g_i$ is a smooth function. Thus $\kappa_i$ and the
normalization $N_i(z)=\kappa_i(z)\cdot z$ are smooth.
Each $H_\sigma$-orbit in this open set has exactly one
representative with $x_i>0$ and $g_i=e_G$. Hence $N_i$
is constant on orbits and descends to the smooth section
$s_i$ for the quotient diffeology.
The relation $s_i=h_{ij}s_j$ gives
$h_{ij}=\kappa_i\kappa_j^{-1}$.  Direct multiplication in
$G\rtimes_\sigma\mathbb Z_2$ yields
\eqref{eq:II-spherical-hij}.  If $z$ is replaced by
$(h,\eta)z$, then
$\eta_i\mapsto\eta\eta_i$ and
$g_i\mapsto h\sigma_\eta(g_i)$, so both expressions
in~\eqref{eq:II-spherical-hij} are unchanged.
They therefore descend to smooth maps on $U_{ij}$.
Finally $h_{ij}h_{jk}=h_{ik}$ follows from
$s_i=h_{ij}s_j=h_{ij}h_{jk}s_k$ and freeness of the
$H_\sigma$-action, giving~\eqref{eq:II-spherical-cocycle}.
\end{proof}

The signs define the double-cover obstruction
\[\alpha_B=[\varepsilon_{ij}]\in\check H^1(B;\mathbb Z_2).\]
For a $D$-numerable principal $H_\sigma$-bundle $P\to X$
with classifying map $f:X\to B$, the same formulas on
$f^{-1}(U_i)$ compute the transitions of $P$.
No connection or differential equation enters this construction.

\subsection{The extension and its coefficient system}
\label{subsec:II-lifting-extension}
Let
\begin{equation}\label{eq:II-central-G}
 1\longrightarrow A\longrightarrow\widehat G
 \stackrel q\longrightarrow G\longrightarrow1
\end{equation}
be a central extension of diffeological groups, where $A$ is
abelian and $q$ is locally smoothly split.  Suppose that the
involution $\sigma$ of $G$ admits a \emph{specified involutive}
lift $\widehat\sigma$ to $\widehat G$.  Write
$\tau_\varepsilon=\widehat\sigma_\varepsilon|_A$.
Then
\begin{equation}\label{eq:II-extended-H}
 1\to A\to\widehat H
 :=\widehat G\rtimes_{\widehat\sigma}\mathbb Z_2
 \stackrel Q\longrightarrow H_\sigma\to1
\end{equation}
is an extension with abelian normal kernel.  It is central
only if $\tau$ is trivial.  Conjugation by any lift of
$(g,\varepsilon)\in H_\sigma$ acts on $A$ as $\tau_\varepsilon$.

For the sign cocycle in \eqref{eq:II-spherical-cocycle},
take the sheaf of smooth $A$-valued functions locally and
glue on $U_{ij}$ by $\tau_{\varepsilon_{ij}}$.  Denote
this sheaf by $\mathscr A_{\alpha_B}$; it is the sheaf
of smooth sections of $E\times_{H_\sigma} A$.
For a pullback bundle $P\to X$ the analogous sheaf is
$\mathscr A_{\alpha(P)}$, with $\alpha(P)=f^*\alpha_B$.
It is not in general a constant coefficient sheaf.

In what follows, transition maps must admit local smooth
lifts through $Q$.  This holds after suitable refinement
when $q$ is locally smoothly split.  To identify Cech
classes with sheaf cohomology we assume that the covers
used admit the required refinements or that the relevant
Cech--sheaf comparison is valid.

\subsection{From a spherical transition function to a lifting obstruction}
\label{subsec:II-spherical-obstruction}
Choose smooth lifts $\widehat g_{ij}$ of the functions
$g_{ij}$ in \eqref{eq:II-spherical-hij}, after refining the
cover if necessary.  Set
$\widehat h_{ij}=(\widehat g_{ij},\varepsilon_{ij})$.
There is no reason for these lifted transitions to obey
$\widehat h_{ij}\widehat h_{jk}=\widehat h_{ik}$.
The discrepancy, computed in the \emph{spherical charts}, is
\begin{equation}\label{eq:II-obstruction-def}
 \boxed{\quad a_{ijk}
 :=\widehat h_{ij}\widehat h_{jk}\widehat h_{ik}^{-1}
 =\widehat g_{ij}\,
 \widehat\sigma_{\varepsilon_{ij}}(\widehat g_{jk})
 \widehat g_{ik}^{-1}\in A.\quad}
\end{equation}
The last equality is understood in the embedded copy
$A\subset\widehat G$ and follows because the signs satisfy
the ordinary cocycle relation.

\begin{theorem}[The obstruction from spherical transitions]
\label{thm:II-obstruction}
The functions $a_{ijk}$ obtained from the spherical
trivializations satisfy
\begin{equation}\label{eq:II-2cocycle}
 a_{ijk}a_{ikl}
 =\tau_{\varepsilon_{ij}}(a_{jkl})a_{ijl}
 \qquad\text{on }U_{ijkl}.
\end{equation}
Under the comparison assumptions above, they determine an
obstruction class
\[
 \beta(P;\widehat H)
 \in H^2(X;\mathscr A_{\alpha(P)}).
\]
It is independent of the lifts and of the local sections,
natural under pullback, and depends on the \emph{chosen
extension}~\eqref{eq:II-extended-H}.
\end{theorem}
\begin{proof}
Expanding a product of three lifted transitions in two
ways gives
\[
 (\widehat h_{ij}\widehat h_{jk})\widehat h_{kl}
 =a_{ijk}a_{ikl}\widehat h_{il}
\]
and
\[
 \widehat h_{ij}(\widehat h_{jk}\widehat h_{kl})
 =\tau_{\varepsilon_{ij}}(a_{jkl})a_{ijl}
 \widehat h_{il},
\]
which proves \eqref{eq:II-2cocycle}.
If the lifts are changed by $A$-valued functions,
$\widehat h'_{ij}=b_{ij}\widehat h_{ij}$, then
\begin{equation}\label{eq:II-change-lifts}
 a'_{ijk}
 =b_{ij}\tau_{\varepsilon_{ij}}(b_{jk})
   a_{ijk}b_{ik}^{-1}.
\end{equation}
Thus the class changes only by a twisted Cech coboundary.
Changing the chosen local sections changes the description
of the same associated coefficient sheaf; comparing both
systems on a common refinement gives the same class.
Every formula is preserved by pullback, proving naturality.
\end{proof}

\begin{theorem}[Vanishing means existence of a lift]
\label{thm:II-lifting}
Assume the stated local-lifting and Cech-descent hypotheses.
Then $\beta(P;\widehat H)=0$ if and only if $P$ admits a
principal $\widehat H$-lift.  The category of local lifts
forms an $\mathscr A_{\alpha(P)}$-banded lifting gerbe with
this class.
\end{theorem}
\begin{proof}
A lift provides $\widehat H$-valued transition functions
satisfying the cocycle identity, so the obstruction vanishes.
Conversely, if the class vanishes, refine the cover and
write
$a_{ijk}=b_{ij}\tau_{\varepsilon_{ij}}(b_{jk})b_{ik}^{-1}$.
Replacing $\widehat h_{ij}$ by
$b_{ij}^{-1}\widehat h_{ij}$ makes the triple products
identically $1$.  These corrected maps glue to a principal
$\widehat H$-bundle projecting to $P$.
The automorphisms of a local lift inducing the identity
on $P$ are precisely smooth sections of the associated
$A$-bundle, and the preceding triple products are the
descent obstruction for local lifts.  This is the asserted
lifting gerbe, with no assertion that arbitrary gerbes or
twists are classified by the construction.
\end{proof}

\subsection{Taking logarithms: the Dixmier--Douady class}
\label{subsec:II-DD-explicit}
Now specialize to $A=U(1)$ and suppose
$\tau_{-1}(z)=\overline z$.  The coefficient system of the
preceding obstruction is therefore the sheaf
$\mathscr U(1)_{\alpha(P)}$ of smooth circle-valued
sections with conjugation along negative transitions.
On sufficiently fine triple overlaps, choose smooth real
functions $\ell_{ijk}$ such that
\begin{equation}\label{eq:II-log-a}
 a_{ijk}=\exp(2\pi i\ell_{ijk}).
\end{equation}
The twisted cocycle relation \eqref{eq:II-2cocycle}
then implies that
\begin{equation}\label{eq:II-DD-integer-cocycle}
 n_{ijkl}
 :=\varepsilon_{ij}\ell_{jkl}
   -\ell_{ikl}+\ell_{ijl}-\ell_{ijk}
 \in\mathbb Z.
\end{equation}
Here a sign $\varepsilon_{ij}=-1$ acts on a chosen logarithm
by multiplication by $-1$.  The expression on the right
is the twisted Cech differential of $\ell$, and its
integer-valuedness follows \emph{directly} by exponentiating
and using \eqref{eq:II-2cocycle}.  Moreover,
$\delta_{\varepsilon}n=0$ follows from
$\delta_{\varepsilon}^2\ell=0$.  Changing logarithms by
integers alters $n$ by an integral twisted coboundary;
changing the lifts $\widehat h_{ij}$ does not alter its
cohomology class.

The sheaf-theoretic version uses the twisted exponential
sequence
\begin{equation}\label{eq:II-exp-sheaf}
 0\longrightarrow\underline{\mathbb Z}_{\alpha(P)}
 \longrightarrow\mathscr R_{\alpha(P)}
 \stackrel{\exp(2\pi i\cdot)}\longrightarrow
 \mathscr U(1)_{\alpha(P)}\longrightarrow0.
\end{equation}
It defines a connecting homomorphism in sheaf cohomology.
\begin{definition}[The spherical lifting DD obstruction]
\label{def:II-DD-obstruction}
When the sheaf model and comparison apply, define
\begin{equation}\label{eq:II-twisted-DD}
 \operatorname{DD}_{\alpha(P)}(P;\widehat H)
 :=\partial\beta(P;\widehat H)
 \in H^3(X;\underline{\mathbb Z}_{\alpha(P)}).
\end{equation}
For a suitable Cech cover, the class is represented by
the explicitly computed integers~\eqref{eq:II-DD-integer-cocycle}.
\end{definition}

\begin{proposition}[Exactly what vanishing detects]
\label{prop:II-DD-vanishing}
Suppose that $\mathscr R_{\alpha(P)}$ is acyclic in degrees
$2$ and $3$, for example when smooth partitions of unity
make it fine and the relevant sheaf theory applies.
Then the connecting homomorphism in degree $2$ is an
isomorphism.  In this case
\[
 \operatorname{DD}_{\alpha(P)}(P;\widehat H)=0
 \quad\Longleftrightarrow\quad
 \beta(P;\widehat H)=0
 \quad\Longleftrightarrow\quad
 P\text{ lifts to }\widehat H.
\]
Without this acyclicity assumption the DD class is still
an image of the lifting obstruction wherever the exponential
sequence is exact, but its vanishing need not imply that
the original lifting obstruction vanishes.
\end{proposition}
\begin{proof}
The long exact sequence associated with
\eqref{eq:II-exp-sheaf} contains
\[
 H^2(X;\mathscr R_{\alpha(P)})\to
 H^2(X;\mathscr U(1)_{\alpha(P)})
 \stackrel\partial\longrightarrow
 H^3(X;\underline{\mathbb Z}_{\alpha(P)})\to
 H^3(X;\mathscr R_{\alpha(P)}).
\]
The outside groups vanish under the assumptions, so
$\partial$ is an isomorphism.  Apply
Theorem~\ref{thm:II-lifting}.
\end{proof}

\begin{remark}[Scope of the construction]
The spherical space produces explicit transitions, not an
extension of $G$ by itself.  If the action $\tau$ on $U(1)$
is trivial, one gets an \emph{untwisted} integral DD class,
even when the sign cover is nontrivial.  Without the
conjugation hypothesis there is no reason for the integral
coefficients in \eqref{eq:II-DD-integer-cocycle} to be
$\mathbb Z_\alpha$.  No claim is made that every element
of $H^3(X;\mathbb Z_\alpha)$ arises from such an extension.
\end{remark}

\subsection{Pullback from the universal spherical space}
\label{subsec:II-universal-obstruction}
On $B=E^{\mathrm{sph}}(G)/H_\sigma$ the same normalizations
$s_i$ and transitions~\eqref{eq:II-spherical-hij} yield a
lifting obstruction on appropriate refinements of the
canonical cover.  Where the required Cech--sheaf comparison
holds, this is a sheaf-cohomology class $\beta_B$.
For a $D$-numerable bundle $P$ classified by $f:X\to B$,
we have, by the explicit pullback formulas,
\begin{equation}\label{eq:II-universal-beta}
 \beta(P;\widehat H)=f^*\beta_B.
\end{equation}
If the exponential-sheaf construction is available on
$B$ and $X$, the same holds for DD classes, including
pullback of the local coefficient sheaf.  In particular,
we do \emph{not} assume without proof that the universal
quotient $B$ has the sheaf-acyclicity properties of a
finite-dimensional paracompact manifold.

\section{Projective realizations and a conditional K-theoretic interpretation}
\label{sec:twisted-k-theory}

The obstruction just obtained is defined entirely by the
spherical quotient, its local sections, and a chosen group
extension.  It does not require a Hilbert space or K-theory.
We explain only when this \emph{same obstruction} can be
recognized in an operator-theoretic realization.

\subsection{A projective realization, when one is supplied}
Fix a separable infinite-dimensional complex Hilbert space
$\mathcal H$ and a conjugation $C$.  Write $U^\pm(\mathcal H)$
for unitary and antiunitary operators and
$PU^\pm(\mathcal H)=U^\pm(\mathcal H)/U(1)$.
There is an extension with conjugation action on the kernel
\begin{equation}\label{eq:II-antiunitary-extension}
 1\to U(1)\to U^\pm(\mathcal H)
 \to PU^\pm(\mathcal H)\to1.
\end{equation}
Assume \emph{in addition} that a homomorphism
\[
 r:H_\sigma\longrightarrow PU^\pm(\mathcal H)
\]
has been chosen, with parity compatible with the projection
$H_\sigma\to\mathbb Z_2$.  We assume that it is smooth
for the diffeologies used for the lifted obstruction, or work
with its continuous realization on a paracompact base where
all associated bundles are topologically locally trivial.
Associated to a principal $H_\sigma$-bundle $P$ is the
projective bundle $\mathcal P_r=P\times_{H_\sigma}
PU^\pm(\mathcal H)$.

\begin{proposition}[Identification of the lifting obstruction]
\label{prop:II-projective-obstruction}
The obstruction to lifting $\mathcal P_r$ through
\eqref{eq:II-antiunitary-extension} is computed by the
triple products of lifted projective transition functions.
If the extension~\eqref{eq:II-extended-H} is isomorphic,
over $H_\sigma$ and with its specified kernel action, to
the pullback of~\eqref{eq:II-antiunitary-extension} along
$r$, then this is exactly $\beta(P;\widehat H)$ in the
smooth sheaf-cohomology setting. In the continuous setting
the equality holds after applying the natural forgetful map
from smooth to continuous cocycles, when such a comparison
is available. Its exponential image then agrees with the
corresponding DD obstruction whenever defined.
\end{proposition}
\begin{proof}
The projective transitions are $r(h_{ij})$ with
$h_{ij}$ given in~\eqref{eq:II-spherical-hij}.  Choose
local unitary or antiunitary lifts $R_{ij}$.
The products $R_{ij}R_{jk}R_{ik}^{-1}$ are scalars;
conjugation by antiunitaries gives the twisted identity
\eqref{eq:II-2cocycle}.  An isomorphism of the pulled-back
extension with $\widehat H$ transports these scalar
discrepancies to~\eqref{eq:II-obstruction-def}.
\end{proof}

\subsection{Relation with complex twisted K-theory}
Suppose that the sign bundle is trivialized and the
projective realization reduces to a principal
$PU(\mathcal H)$-bundle on a paracompact topological base.
Then its \emph{ordinary} integral DD obstruction is the
usual characteristic invariant of a complex projective
Hilbert bundle.  That projective bundle can be used to
define a twisted complex K-group via a stable Fredholm
bundle, as in~\cite{AtiyahSegal2004,Rosenberg1989}.
The present article does not need the Fredholm definition,
and it does not assert a classification of K-theory
transformations or twists.  If the projective bundle is
trivialized, its twisted K-theory reduces, after that
choice, to ordinary complex K-theory.

For nontrivial $\alpha(P)$, the same lifting calculation
recognizes the obstruction to a projective
\emph{unitary/antiunitary} lift if a compatible realization
$r$ is available.  Turning this into a particular
K-theory theory requires a further choice of an appropriate
operator or spectrum model.  This goes beyond the lifting
obstruction proved here; no canonical identification is
claimed from the pair $(\alpha,\operatorname{DD}_\alpha)$
alone.

\begin{remark}
Other varieties of topological twisting may involve
additional grading or real structures.  They are not
addressed by the spherical lifting calculation above.
The contribution of this article is the explicit
obstruction and its functoriality, not a classification of
all possible K-theoretic twists.
\end{remark}

\section{Descent of differential and geometric structures}
\label{sec:descent}

\subsection{The quotient to which a structure descends}
Let $E=E^{\mathrm{sph}}(G)$ and let $H$ be one of $G$,
$G\times\mathbb Z_2$, or $H_\sigma=G\rtimes_\sigma\mathbb Z_2$
acting as in Sections~4--5. Write $q_H:E\to E/H$.
The destinations $E/G$, $E/(G\times\mathbb Z_2)$, and
$E/H_\sigma$ are generally different diffeological spaces.
In particular, invariance under $G$ alone does not supply
descent to the latter two quotients.

For a smooth principal $H$-bundle, the vertical directions
are generated by the Lie algebra $\mathfrak g$ of the identity
component. A differential form $\omega\in\Omega^k(E)$ is called
$H$-\emph{basic} when
\[
 h^*\omega=\omega\quad(h\in H),\qquad
 \iota_{\xi_E}\omega=0\quad(\xi\in\mathfrak g).
\]
The sign component must satisfy invariance but introduces
no additional infinitesimal vertical vectors.

\begin{theorem}[Plotwise descent of basic forms]\label{thm:II-basic-descent}
If $q_H:E\to B$ is a diffeological principal $H$-bundle
admitting local smooth sections, then pullback gives an
isomorphism of differential graded algebras
\begin{equation}\label{eq:II-basic-forms}
 q_H^*:\Omega^\bullet(B)\stackrel{\simeq}\longrightarrow
 \Omega^\bullet_{\mathrm{basic},H}(E).
\end{equation}
\end{theorem}
\begin{proof}
The pullback of a form on $B$ is invariant because
$q_H\circ h=q_H$, and horizontal because $q_H$ is constant
on each orbit. Injectivity holds since $q_H$ is a subduction:
a form vanishing after pullback vanishes on plots which locally
lift to $E$, hence on every plot of $B$.
For surjectivity let $\omega$ be basic. For a plot $p:U\to B$,
locally choose a smooth lift $\widetilde p:V\to E$ and assign
$\eta_p|_V=\widetilde p^*\omega$. If $\widetilde p'$ is another
lift, then $\widetilde p'=h\cdot\widetilde p$ for a smooth
$h:V\to H$. In a local trivialization the derivative of this
map is the derivative of a fixed group action plus a vertical
term arising from $dh$. Invariance handles the former and
horizontality kills every alternating term containing the
latter. Thus $\widetilde p'^*\omega=\widetilde p^*\omega$.
These local assignments glue, commute with smooth changes of
plots, and therefore define a diffeological form $\eta$ on
$B$. By construction $q_H^*\eta=\omega$. Exterior differentiation
commutes with pullback; Cartan's identities show that basic
forms are stable under $d$, establishing the DGA statement.
\end{proof}

\subsection{Connections and characteristic forms}
For a principal $H_\sigma$-bundle the barycentric connection
of Part~I is a connection with values in $\mathfrak g$.
Its connection 1-form on $E$ is equivariant rather than
horizontal; it therefore does not descend as an ordinary
1-form on $B$. Its curvature is horizontal and equivariant.
If $P$ is a continuous symmetric invariant polynomial on
$\mathfrak g$ under the entire $H_\sigma$-adjoint action,
$P(F^r)$ is $H_\sigma$-basic and descends by
Theorem~\ref{thm:II-basic-descent}. The Bianchi identity
and invariance of $P$ give closedness, so the descended
form represents the corresponding Chern--Weil class.
This statement also applies to Fr\'echet--Lie groups when the
connection, curvature and invariant polynomial are smooth
and continuous in the adopted differential calculus.

\subsection{Riemannian descent on support strata}
Fix a finite nonempty support set $J$, and denote by
$E_J^\circ$ the corresponding stratum. The product metric of
Part~I has a nondegenerate restriction to $E_J^\circ$ when
$G$ carries a suitable (possibly weak) invariant metric.
For a finite-dimensional $G$, the diagonal $G$-action is
isometric with the appropriate left-invariant convention.
For the semidirect action, one must \emph{also} require invariance
under the involution $\sigma$: left invariance of the metric on
$G$ alone does not imply this. The following averaging argument
supplies compatible invariant data under the hypotheses of Part~I.

\begin{proposition}[Averaging the metrics and the kernel]
\label{prop:II-sigma-metric}
Let $\sigma$ be a smooth involutive automorphism of a
Fr\'echet--Lie group $G$. If $h$ is a smooth positive-definite
left-invariant (possibly weak) metric on $G$, then
\[
 h^{\sigma}:=\tfrac12(h+\sigma^*h)
\]
is again smooth, positive-definite and left-invariant, and is
$\sigma$-invariant. If, in addition, $G$ admits a smooth,
normalized, left-invariant, strictly positive-definite kernel
$K(g,k)$ as in Part~I, then
\[
 K^{\sigma}(g,k)
 :=\tfrac12\bigl(K(g,k)+K(\sigma(g),\sigma(k))\bigr)
\]
has the same properties and is invariant under the diagonal
$\sigma$-action. Thus the conditional global internal-tangent
metric of Part~I may be chosen invariant under
$G\rtimes_{\sigma}\mathbb Z_2$.
\end{proposition}
\begin{proof}
Because $\sigma$ is an automorphism, pulling back a
left-invariant metric by $\sigma$ again gives a
left-invariant metric. Taking the average preserves smoothness,
symmetry and strict positive definiteness, and
$\sigma^2=\mathrm{id}$ makes the average invariant.
For the kernel, both $K$ and $K\circ(\sigma\times\sigma)$
are smooth and normalized. For any finite set of pairwise
distinct $g_i$, the points $\sigma(g_i)$ are again distinct,
so each Gram matrix is strictly positive definite. Their average
is strictly positive definite. Left invariance follows from
$\sigma(ag)=\sigma(a)\sigma(g)$, and the average is
$\sigma$-invariant. Part~I's metric construction, evaluated
on activated branches and active group directions, is therefore
invariant under both $G$ and the sign-automorphism action.
\end{proof}

If the orbit tangent bundle has a smooth orthogonal complement,
the metric restricts to the horizontal complement and yields
a quotient metric by the standard Riemannian submersion rule:
\[
 g_B(u,v)=g_E(u^{\mathrm{hor}},v^{\mathrm{hor}}).
\]
An additional $\mathbb Z_2$-quotient is legitimate when the
sign action is also isometric; its derivative preserves the
horizontal distribution. In a weak Fr\'echet metric, an
orthogonal complement need not exist or be a smooth bundle.
Consequently this construction is conditional and does not
assert a global Riemannian metric across changes of support.

\subsection{What is needed for Hodge and Dirac descent}
For a finite-dimensional Riemannian support stratum, let
$D_J=d+\delta_J$ be its de Rham--Dirac operator. The
existence of $D_J$ does not entail a global operator on $E$:
Part~I supplies counterexamples to automatic gluing of
stratumwise Laplacians. There are two further, independent
issues when passing to a quotient:
(i) the operator and its domain must be invariant under the
entire acting group; (ii) it must send the chosen basic
subspace into itself.

\begin{proposition}[Conditional operator descent]
\label{prop:II-operator-descent}
Let $\mathcal D$ be an $H$-stable space of smooth sections
of an $H$-equivariant vector bundle on $E$, and let
$D:\mathcal D\to\Gamma(E,V)$ be an operator. Assume
$D$ preserves the subspace $\mathcal D_{\mathrm{bas}}$
identified with sections of a bundle $\bar V$ over $B=E/H$.
Then there is a unique operator
$\bar D$ on the corresponding space of sections over $B$
such that $q_H^*(\bar D s)=D(q_H^*s)$.
If $D$ is a closed unbounded operator, this conclusion at
the Hilbert-space level additionally requires compatible
$L^2$-structures and that pullback identify the relevant
closed operator domains.
\end{proposition}
\begin{proof}
For $s$ on $B$, the pullback lies in $\mathcal D_{\mathrm{bas}}$;
by assumption $D(q_H^*s)$ is again basic and hence descends
to a unique section. This defines $\bar D$. The equality
with $q_H^*$ makes it unique because pullback through a
subduction is injective on sections. No conclusion about
closures follows algebraically: the additional Hilbert-space
hypotheses ensure that the graph norm, domain and adjoints
are transported correctly.
\end{proof}

A spinorial operator requires, moreover, an equivariant
Clifford module, an equivariant compatible connection and,
for a spin Dirac operator, a suitable spin or spin$^c$
structure on the corresponding finite-dimensional geometry.
Neither an $H$-equivariant Clifford action nor invariance
of its connection follows from the existence of the spherical
coordinates. Thus the descent claims for Hodge and Dirac
operators are to be read under the conditions of
Proposition~\ref{prop:II-operator-descent}, not as automatic
features of the whole diffeological quotient.

\section{Equivariance defects of differential operators}
\label{sec:G-stability-Hodge-Dirac}

\subsection{The group cocycle of an operator}
Suppose a group $H$ acts on a space of sections by
invertible operators $U_h$, and that a common invariant
linear domain $\mathcal D$ is available for all expressions
below. For an operator $D:\mathcal D\to\mathcal D$ set
\begin{equation}\label{eq:II-equiv-defect}
 \mathfrak C_D(h)=U_hDU_h^{-1}-D.
\end{equation}
This vanishes exactly when $D$ commutes with the group action.

\begin{proposition}[Defect identity]\label{prop:II-defect-identity}
For $h,k\in H$ one has
\begin{equation}\label{eq:II-defect-cocycle}
 \mathfrak C_D(hk)=\mathfrak C_D(h)
 +U_h\mathfrak C_D(k)U_h^{-1}.
\end{equation}
Thus $\mathfrak C_D$ is a group 1-cocycle with values in the
operator module under conjugation. It is a 1-\emph{coboundary}
there, not automatically a new nontrivial cohomology class.
\end{proposition}
\begin{proof}
Insert $D-D$ between
$U_h(U_kDU_k^{-1})U_h^{-1}$ and $D$:
\[
 U_hU_kDU_k^{-1}U_h^{-1}-D
 =U_h(U_kDU_k^{-1}-D)U_h^{-1}
 +(U_hDU_h^{-1}-D).
\]
This is \eqref{eq:II-defect-cocycle}; by its definition it is
the coboundary of the zero-cochain $D$ under conjugation.
\end{proof}

If the Lie algebra action differentiates to
$\dot U:\mathfrak h\to\operatorname{End}(\mathcal D)$,
its infinitesimal defect is $[\dot U(X),D]$ (with the
chosen conjugation convention). A Jacobi-identity calculation
gives its Chevalley--Eilenberg 1-cocycle identity. Again, in
the full operator module this is a coboundary.
A genuinely nontrivial anomaly requires an additional
quotient, central extension, regularization, or restriction
of the coefficient module; none is supplied by
\eqref{eq:II-equiv-defect} alone.

\subsection{Basic forms and adjoints}
For $D=d$ on ordinary forms the diffeological exterior
derivative is natural, so $U_h d=dU_h$ and the defect
vanishes. If the action preserves a Riemannian metric,
orientation where needed, and the analytic realization,
it commutes with the codifferential and the Hodge Laplacian
on the same invariant domain. If the action is not isometric,
$\delta$ need not be equivariant. Even if it is equivariant,
$\delta$ need not preserve horizontal basic forms: an
additional condition on fibers, the mean-curvature term or
an adapted quotient Hodge operator may be required.
The global absence of an a priori Hodge Laplacian on $E$
remains independent of this equivariance question.

\subsection{Clifford and spinorial operators}
Let $(V,c,\nabla)$ be an $H$-equivariant Clifford module on
a chosen smooth stratum or finite-dimensional model.
When the group action preserves the metric, Clifford action,
connection and analytic domain, the associated Dirac operator
commutes with $U_h$; this follows in local orthonormal frames
from the covariance of Clifford multiplication and covariant
differentiation. Conversely, dropping these hypotheses may
produce a nonzero defect. The formula
\eqref{eq:II-defect-cocycle} records its transformation under
the group, but does not construct a central extension.

For a principal bundle $P\to X$ classified by a spherical map,
pulling back equivariant geometric data also pulls back
equivariance identities. Pullback of an unbounded operator is
not a formal operation without a selected bundle, connection,
operator family and domain; we do not identify the pullback of
a defect with a globally defined Dirac anomaly unless these
analytic choices have been made.

\section{Chern--Weil and transgression data associated with defects}
\label{sec:defect-classes}

\subsection{A connection difference is not itself a curvature}
Let $V\to Y$ be a smooth finite-rank vector bundle, or a
bundle with structure Lie algebra admitting continuous
invariant polynomials and a compatible calculus. Suppose
$h\in H$ acts by bundle automorphisms and $\nabla$ is a
connection. Define the connection defect
\begin{equation}\label{eq:II-connection-defect}
 \mathfrak a_h=h^*\nabla-\nabla\in
 \Omega^1(Y;\operatorname{End}(V)).
\end{equation}
Since differences of connections are tensorial, this is a
well-defined $\operatorname{End}(V)$-valued one-form. It
obeys the affine cocycle rule
\[
 \mathfrak a_{hk}=\mathfrak a_k+k^*\mathfrak a_h
\]
when $(hk)^*=k^*h^*$; one may equivalently reverse the
group convention. It does not follow that
$d\mathfrak a_h+\mathfrak a_h\wedge\mathfrak a_h$
is the curvature of an intrinsically defined new connection.

\subsection{A rigorous transgression formula}
Let $P$ be an invariant continuous symmetric $r$-linear
polynomial. Consider two connections
$\nabla_0,\nabla_1=\nabla_0+a$ and the affine path
$\nabla_t=\nabla_0+ta$. Let $F_t$ be its curvature. Define
\begin{equation}\label{eq:II-CS}
 \operatorname{CS}_P(\nabla_0,\nabla_1)
 =r\int_0^1 P(a,F_t,\ldots,F_t)\,dt
 \in\Omega^{2r-1}(Y).
\end{equation}

\begin{theorem}[Transgression identity]\label{thm:II-CS}
The Chern--Simons form satisfies
\begin{equation}\label{eq:II-CS-d}
 d\operatorname{CS}_P(\nabla_0,\nabla_1)
 =P(F_1^r)-P(F_0^r).
\end{equation}
If $\nabla_1=h^*\nabla_0$ and $P$ is invariant under the
bundle action, the right-hand side equals
$h^*P(F_0^r)-P(F_0^r)$. Thus the de Rham cohomology class
of the characteristic form is invariant under the action.
A Chern--Simons form is closed only when its right-hand side
vanishes as a form, not merely in cohomology.
\end{theorem}
\begin{proof}
The curvature variation of $\nabla_t$ is
$\partial_tF_t=d_{\nabla_t}a$. Hence invariance of $P$
and the Bianchi identity give
\[
 \frac d{dt}P(F_t^r)=rP(d_{\nabla_t}a,F_t^{r-1})
 =r\,dP(a,F_t^{r-1}).
\]
Integration in $t$ proves \eqref{eq:II-CS-d}. When
$\nabla_1=h^*\nabla_0$, functoriality of curvature gives
$F_1=h^*F_0$ and the claimed formula follows. In particular
$P(F_1^r)$ and $P(F_0^r)$ represent the same de Rham
class, independently of any new defect-characteristic class.
\end{proof}

\subsection{Scope of defect characteristic constructions}
Additional secondary invariants may be defined from
\eqref{eq:II-CS} after specifying boundary conditions,
cycles, group-cohomological descent or differential-cohomology
data. Such constructions require checking closedness and
independence of auxiliary choices separately. One cannot
infer a nonzero characteristic class from an operator
commutator or an equivariance defect alone. In particular,
a Schwinger cocycle is a Lie-algebra 2-cocycle obtained under
restricted-operator hypotheses (next section), not an
automatic application of \eqref{eq:II-CS}.

\section{Restricted operators, Schwinger cocycles and lifting gerbes}
\label{sec:defect-schwinger-gerbe}

\subsection{The restricted Lie algebra}
Let $\mathcal H=\mathcal H_+\oplus\mathcal H_-$ be a polarized
complex Hilbert space, with orthogonal projections $P_\pm$
and grading operator $\epsilon=P_+-P_-$. Denote the
Hilbert--Schmidt ideal by $\mathcal L^2(\mathcal H)$.
The restricted Lie algebra is
\[
 \mathfrak{gl}_{\mathrm{res}}(\mathcal H)
 =\{A\in\mathcal B(\mathcal H):[\epsilon,A]\in\mathcal L^2\}.
\]
Its natural topology is given by the bounded-operator norm
together with the Hilbert--Schmidt norm of $[\epsilon,A]$.
Write the off-diagonal blocks of $A$ as
$A_{+-}=P_+AP_-$ and $A_{-+}=P_-AP_+$.

\begin{definition}[Schwinger cocycle]\label{def:II-Schwinger}
For $A,B\in\mathfrak{gl}_{\mathrm{res}}(\mathcal H)$ set
\begin{equation}\label{eq:II-schwinger-block}
 c_{\mathrm{Sch}}(A,B)=
 \operatorname{Tr}(A_{-+}B_{+-}-B_{-+}A_{+-}).
\end{equation}
Equivalently,
\[
 c_{\mathrm{Sch}}(A,B)=\frac18\operatorname{Tr}
 \bigl(\epsilon[[\epsilon,A],[\epsilon,B]]\bigr).
\]
\end{definition}
The traces exist because both off-diagonal blocks are
Hilbert--Schmidt and products of two Hilbert--Schmidt
operators are trace class.

\begin{proposition}[Cocycle identity]\label{prop:II-schwinger}
The bilinear form $c_{\mathrm{Sch}}$ is continuous,
alternating, and satisfies
\begin{equation}\label{eq:II-CE-schwinger}
 c_{\mathrm{Sch}}([A,B],C)+
 c_{\mathrm{Sch}}([B,C],A)+
 c_{\mathrm{Sch}}([C,A],B)=0.
\end{equation}
Consequently it defines a continuous central extension
of the \emph{Lie algebra}
$\mathfrak{gl}_{\mathrm{res}}(\mathcal H)$.
\end{proposition}
\begin{proof}
Continuity follows from
$|\operatorname{Tr}(ST)|\le\|S\|_2\|T\|_2$
for Hilbert--Schmidt $S,T$. Alternation is immediate from
the displayed definition. Expand the off-diagonal block
of a commutator:
\[
 [A,B]_{-+}=A_{--}B_{-+}+A_{-+}B_{++}
 -B_{--}A_{-+}-B_{-+}A_{++},
\]
and similarly for its $+-$ block. Substitute these
expressions into the cyclic sum in \eqref{eq:II-CE-schwinger}.
Terms with a fixed diagonal block cancel in pairs using
cyclicity of the trace when a bounded operator multiplies
a trace-class product of off-diagonal Hilbert--Schmidt
blocks. This proves \eqref{eq:II-CE-schwinger}.
The vector space $\mathfrak{gl}_{\mathrm{res}}\oplus\mathbb C$
with bracket
$[(A,z),(B,w)]=([A,B],c_{\mathrm{Sch}}(A,B))$
is therefore a Lie algebra; centrality of the second summand
is immediate.
\end{proof}

\subsection{Pullback and the integrability condition}
Suppose $\rho:\mathfrak g\to\mathfrak{gl}_{\mathrm{res}}$
is a continuous Lie algebra homomorphism. Then
\[
 c_\rho(X,Y)=c_{\mathrm{Sch}}(\rho(X),\rho(Y))
\]
is a continuous Lie algebra 2-cocycle. This follows by
substituting $\rho([X,Y])=[\rho(X),\rho(Y)]$ into
\eqref{eq:II-CE-schwinger}. If instead one uses the
operator defects $X\mapsto[\rho(X),D]$, this map is in
general \emph{not} a Lie algebra homomorphism; its Schwinger
pairing need not satisfy a 2-cocycle identity. A further
calculation or hypothesis is needed before calling the
result an anomaly class.

A Lie algebra cocycle does not automatically integrate to a
smooth central extension of a given Lie group. The period
homomorphism and any flux obstruction must satisfy the
appropriate integrability conditions, as in the theory of
infinite-dimensional Lie group extensions \cite{Neeb2004,Neeb2007}.
When a smooth central group extension
\[
 1\to U(1)\to\widehat G\to G\to1
\]
has been constructed, its lifting gerbe for a principal
$G$-bundle is exactly the untwisted instance of
Section~\ref{sec:gerbes}. If an involution of $G$ lifts to
one of $\widehat G$ and acts on the central circle, the
same construction gives the corresponding twisted band.
The existence of that lifted involution must be checked,
not inferred from a Lie algebra cocycle alone.

\subsection{Naturality on spherical classifying spaces}
For a Lie group $G$ covered by the classifying theorem of
Part~I, the central extension above gives a lifting class
on the spherical classifying quotient $E^{\mathrm{sph}}(G)/G$.
For the semidirect quotient, one obtains a class only
when the extension and its involution are compatible as in
\eqref{eq:II-extended-H}. Its pullback to any classified
principal bundle equals that bundle's lifting obstruction
by Theorem~\ref{thm:II-obstruction}.
This functorial relationship is independent of the existence
of a global Dirac operator on $E^{\mathrm{sph}}(G)$.

\section{The loop-group example}
\label{sec:loop-example}

\subsection{Hardy polarization and restricted multiplication}
Let $K$ be a compact Lie group and let
$\rho:K\to U(V)$ be a finite-dimensional unitary
representation. Set
\[
 \mathcal H=L^2(S^1,V),\qquad
 \mathcal H_+=\overline{\operatorname{span}}\{z^n v:n\ge0\},
 \quad\mathcal H_-=\overline{\operatorname{span}}\{z^nv:n<0\}.
\]
For $X\in L\mathfrak k=C^\infty(S^1,\mathfrak k)$,
let $M_X$ be multiplication by $d\rho(X(\theta))$.
Smoothness implies rapid Fourier decay. The matrix entries
of $[\epsilon,M_X]$ vanish except when the Fourier shift
crosses the polarization, and
\[
 \|[\epsilon,M_X]\|_2^2
 \le C\sum_{n\in\mathbb Z}|n|\,
 \|\widehat{d\rho(X)}(n)\|_{\mathrm{HS}}^2<\infty.
\]
Thus $X\mapsto M_X$ takes values in
$\mathfrak{gl}_{\mathrm{res}}$ and is a continuous Lie algebra
homomorphism.

\begin{theorem}[The loop Schwinger term]\label{thm:II-loop}
With the block convention of \eqref{eq:II-schwinger-block},
\begin{equation}\label{eq:II-loop-cocycle}
 c_{\mathrm{Sch}}(M_X,M_Y)
 =\frac{1}{2\pi i}\int_0^{2\pi}
 \operatorname{tr}_V\bigl(d\rho(X(\theta))
 \tfrac d{d\theta}d\rho(Y(\theta))\bigr)\,d\theta.
\end{equation}
The overall sign changes if the polarization or the
ordering of the blocks in \eqref{eq:II-schwinger-block} is
reversed.
\end{theorem}
\begin{proof}
By bilinearity and rapid decay it is enough first to take
trigonometric monomials
$X(\theta)=e^{im\theta}A$ and
$Y(\theta)=e^{in\theta}B$ in the complexified loop algebra.
Only the case $m+n=0$ contributes to the trace. When
$m>0$ and $n=-m$, multiplication by $e^{im\theta}$
moves exactly $m$ negative Fourier modes into the positive
subspace, while $e^{-im\theta}$ moves the corresponding
$m$ modes back. Hence the first term of
\eqref{eq:II-schwinger-block} vanishes and the second
contributes $-m\operatorname{tr}_V(AB)$.
The right-hand side of \eqref{eq:II-loop-cocycle}
for the same pair is
\[
 (2\pi i)^{-1}\int_0^{2\pi}
 \operatorname{tr}_V(e^{im\theta}A(-im)e^{-im\theta}B)
 \,d\theta=-m\operatorname{tr}_V(AB).
\]
The case $m<0$ follows by antisymmetry. Summing Fourier
series is justified by rapid decay and trace-class
convergence of the off-diagonal products.
\end{proof}

\subsection{Central extensions and gerbes}
The invariant bilinear form
\[
 B_\rho(A,B)=\operatorname{tr}_V(d\rho(A)d\rho(B))
\]
fixes the level of the resulting loop algebra cocycle.
For the standard integrable levels and the relevant smooth
loop group, the corresponding Kac--Moody central extension
exists \cite{PressleySegal1986}. Its bundle-lifting
obstruction is a class in $H^3(X;\mathbb Z)$, provided the
base is within the hypotheses of Section~\ref{sec:gerbes}.
This class is the pullback of the universal lifting-gerbe
class on $B(LK)$, which may be modeled by the spherical
quotient $E^{\mathrm{sph}}(LK)/LK$ under the classifying
hypotheses. For the \emph{double} or semidirect quotient,
one first requires the extension to be compatible with the
corresponding sign action. The Schwinger term does not by
itself produce a global Dirac defect on the entire spherical
Milnor space.

\section{Current groups, spin geometry and non-spin cases}
\label{sec:current-groups-spin}

\subsection{The natural Dirac polarization}
Let $M$ be a compact closed Riemannian spin manifold of
dimension $m$, let $K$ be compact and let $V$ carry a
finite-dimensional unitary representation of $K$. Put
\[
 \mathcal H=L^2(M,S\otimes V),\qquad
 G_M=C^\infty(M,K),\qquad
 \mathfrak g_M=C^\infty(M,\mathfrak k).
\]
The Dirac operator $D_M$ defines a spectral polarization,
with a fixed finite-dimensional convention for the kernel.
For $X\in\mathfrak g_M$, multiplication by $d\rho(X)$
is bounded on $\mathcal H$.

\begin{proposition}[The Schatten obstruction]\label{prop:II-Schatten}
The operator
$[\operatorname{sign}(D_M),M_X]$
is pseudodifferential, of order \emph{at most} $-1$.
For $m\ge2$ its principal symbol of order $-1$ is generically
nonzero for nonconstant currents. It belongs to the Schatten
class $\mathcal L^p$ for every $p>m$, but is not in general
Hilbert--Schmidt when $m\ge2$. In contrast, in the
one-dimensional Hardy polarization, multiplication by a smooth
loop has smoothing off-diagonal blocks. Therefore the ordinary
restricted-Hilbert
Schwinger cocycle of Section~\ref{sec:defect-schwinger-gerbe}
is not automatically defined on $C^\infty(M,K)$ in higher
dimensions.
\end{proposition}
\begin{proof}
Away from the zero section, the principal symbol of
$\operatorname{sign}(D_M)$ is the normalized Clifford
symbol $c(\xi)/|\xi|$ (up to the Dirac sign convention).
Since currents act on $V$ and commute with Clifford
multiplication, the order-zero commutator symbol vanishes.
Its order-$-1$ symbol is, up to the conventional factor,
\[
 \frac1i\sum_j
 \partial_{\xi_j}\bigl(c(\xi)/|\xi|\bigr)
 \partial_{x_j}d\rho(X).
\]
For $m\ge2$ this is generically nonzero. The standard
Schatten estimate for an operator of order $-s$ on a closed
$m$-manifold gives membership in $\mathcal L^p$ for
$sp>m$. With $s=1$ one obtains $p>m$; if the displayed
principal symbol is nonzero, the singular-value asymptotics
preclude Hilbert--Schmidt membership for $m\ge2$.
For $m=1$, by contrast, the normalized Clifford symbol is
constant on each component $\xi>0$ and $\xi<0$ of the
punctured cotangent line, so its $\xi$-derivatives vanish
there. In the standard Hardy/Fourier model the off-diagonal
blocks of the commutator have rapidly decaying matrix entries,
and are in particular Hilbert--Schmidt, as computed in
Section~\ref{sec:loop-example}. The ordinary Schwinger
trace is therefore legitimate for smooth loops but not
automatically in higher dimensions.
\end{proof}

In dimension one, Proposition~\ref{prop:II-Schatten}
recovers the Hilbert--Schmidt mechanism and the loop
example of Section~\ref{sec:loop-example}. In higher
dimensions one may study higher Schatten ideals,
renormalized traces, pseudodifferential cocycles and
anomaly gerbes. These require separate cocycle and
integrability proofs, and are not established merely by
the presence of a Dirac operator on $M$.

\subsection{Spin$^c$ and non-spin geometry}
On an oriented Riemannian manifold, a spin structure
requires $w_2(TM)=0$, while a spin$^c$ structure exists
precisely when the integral Bockstein
\[
 W_3(TM)=\beta(w_2(TM))\in H^3(M;\mathbb Z)
\]
vanishes. The latter is an obstruction to a spin$^c$
\emph{lifting problem}; it is not automatically the
Schwinger class of an arbitrary current group.
If there is no spin$^c$ structure, one may still work
with the de Rham--Dirac operator $d+\delta$ on differential
forms, or with suitable bundles of Clifford modules.
These do not provide a canonical irreducible spinor
polarization with the same determinant-line or current
algebra extension as in the spin case. Any twisted spinorial
construction needs the corresponding gerbe or module data.
For nonorientable manifolds one must distinguish the
appropriate $\mathrm{Pin}$ or $\mathrm{Pin}^c$ problem;
$W_3(TM)$ alone is not a universal replacement.

\subsection{Comparison and conclusions}
The three cases have sharply different scopes.
On the circle, the restricted-operator trace and the
integrable loop-group extension yield an explicit
Schwinger cocycle and lifting gerbe. On a higher-dimensional
spin manifold, Dirac geometry exists but generally does not
satisfy the Hilbert--Schmidt restriction needed for that
cocycle. On a non-spin manifold, one must first specify
whether one uses a spin$^c$-lifting gerbe, a Clifford-module
Dirac operator, or a different anomaly construction.

Accordingly, the spherical Milnor construction supplies
a common smooth classifying setting for these principal
bundles and their existing lifting data. It does not, on
its own, turn an arbitrary operator defect into a gerbe,
produce a K-theory twist without a model, or remove
Schatten and integrability obstructions. The preceding
results identify exactly the extra geometric and analytic
choices required for each of these applications.

\end{document}